\documentclass{amsart}

\usepackage{amssymb, amsmath, latexsym, a4}
\usepackage[latin1]{inputenc}
\usepackage[all]{xy}
\usepackage{enumitem}

\def\1{^{-1}}
\def\id{{\sf id}}

\newtheorem{De}{Definition}

\newtheorem{Pro}[De]{Proposition}
\newtheorem{Le}[De]{Lemma}
\newtheorem{Co}[De]{Corollary}

\newtheorem*{Ex*}{Examples}
\newtheorem*{Example*}{Example}

\def\xto#1{\xrightarrow[]{#1}}
\def\d{\partial}
\def\1{^{-1}}
\def\im{\sf Im}
\def\ker{\sf ker}
\def\al{{\alpha}}
\def\bu{\bullet}
\def\bb{{\beta}}
\def\dd{{\delta}}

\begin{document}
\title{On the nonabelian cohomology with coefficients in a crossed module}
\author{MAriam Pirashvili}

\maketitle

\begin{abstract}
This paper is concerned with the nonabelian cohomology of groups with coefficients in crossed modules. These objects were introduced by Dedecker and studied by Breen, Borovoi, Noohi and many others. In this paper we study several important objects that are classified by these cohomologies and for them we develop a theory similar to the Schreier obstruction theory of group extensions.
\end{abstract}

\section{Introduction}
%The second section Schreier theory nonabelian extensions of groups. Little new. Most Dedecker's work. To the best of our knowledge, this is not explicitly written in any of his papers. For the case when the crossed module is $M\to Aut(M)$, MacLane. Except equivalence. For the case of a general crossed module $M\to L$, Dedecker. In terms of a cosimplicial crossed module.

This paper is concerned with a variant of nonabelian cohomology which generalises and extends classical nonabelian cohomology of groups \cite{serre}. The main object of study is the nonabelian cohomology of groups with coefficients in crossed modules. More specifically, given a nonabelian group $G$ which acts in a particular sense on an crossed module $M\to L$, a nonabelian cohomology $H^i(G, M\to L)$ can be defined for $i=0,1,2$. These objects were introduced by Dedecker in the 1960's \cite{dedecker} and have since been studied by Borovoi \cite{borovoi}, Noohi \cite{noohi} and many others. These objects can be obtained as a special case of the more general theory of nonabelian cohomotopy of cosimplicial crossed modules. In Section \ref{Cohomotopy} all these definitions are recalled.

Section $3$ is a generalisation of the Schreier theory of group extensions, which is based on Dedecker's work \cite{dedecker}. In that paper, Dedecker introduced the notion of an $(M\to L)$ extension of a group $G$ by $M$. He also introduced two versions of the second cohomology of $G$ with coefficients in $(M\to L)$, denoted by $H^2(G, M\to L)$ and $\widehat{H}^2(G,M\to L)$. Two equivalence relations on the set of extensions are introduced in this paper. This seems to be missing in Dedecker's paper. The resulting congruence classes are shown to be isomorphic to the cohomologies $H^2(G, M\to L)$ and $\widehat{H}^2(G,M\to L)$. In the case when the crossed module is $M\to Aut(M)$, the case of $\widehat{H}^2(G,M\to Aut(M))$ is equivalent to the Schreier theory,  while the case $H^2(G,M\to Aut(M))$ seems to be new.

Section $4$ is a special case of the general theory of non-abelian cohomology of toposes by Grothendieck and Giraud \cite{giraud}, when the topos is the category of $G$-sets for a fixed group $G$. Instead of considering gerbes, we use Duskin's approach of considering the related objects called bouquets \cite{duskin}. In the case of the topos being $G$-sets, the bouquets are referred to as $G$-equivariant bouquets and are $G$-equivariant groupoids $\Gamma$, which are non-empty and connected.

For a crossed module $M\to L$, we generalise this notion and introduce $G$-equivariant bouquets defined over $M\to L$. The case of standard $G$-bouquets is obtained when we use the standard crossed module $M\to Aut(M)$, where $M$ is the group of automorphisms of a chosen object of $\Gamma$. The main result of this section is showing that such bouquets over $M\to L$ are classified by the second cohomology of $G$ with coefficients in $M\to L$. Our proof differs from Duskin's approach. It is direct and establishes a connection between bouquets and  the theory of   $(M\to L)$-extensions developed in  Section $4$.

The aim of Section $5$ is to describe $H^1(G,M\to L)$ in terms of torsor-like structures. The motivation for our definition of bitorsors over $M\to L$ comes from an old work by Breen \cite{breen}, where he defined bitorsors, which in our setting become bitorsors over $M\to Aut(M)$.

Our first result of this section is proving a bijection between bitorsors $M\to L$ and $H^1(G,M\to L)$. Breen's result can be obtained as a special case of this when $L = Aut(M)$. The second topic of this section is related to obstruction theory. In more detail, given a crossed extension
$$0\to A\to M\xto{\d} L\xto{\pi}\, Q\to 1,$$
for any $(M\xto{\d} L)$-bitorsor $P$, we construct an invariant $\pi_*(P)\in H^0(G,Q)$. We answer the following question: For a given element $a\in H^0(G,Q)$, does there exist an $(M\xto{\d} L)$-bitorsor $P$ such that $\pi_*(P)=a$? To answer the question, we construct an element $o(a)\in H^2(G,A)$ such that $o(a)=0$ iff $\pi_*(P)=a$. The construction of $o(a)$ uses the results of the previous section about equivariant bouquets. This an analogue of the classical Schreier-Eilenberg-Maclane obstruction theory for group extensions, but taking place one dimension lower. In the last part of this section, we study a weaker notion of bitorsors and we show that under some assumptions on $L$ the two notions are equivalent. This weaker version is intuitively more closely related to Breen's definition of bitorsors \cite{breen}.

\section{Cohomotopy of cosimplicial crossed modules}\label{Cohomotopy}

\subsection{Crossed modules}  A \emph{crossed module} is a group homomorphism $\d:M\to L$ (shortly $\d$), together with an action of the group $L$ on the group $M$, such that
$$\d(^xm)=x\d(m)x\1$$
$$^{\d(n)}m=nmn\1$$
Here $x\in L$, $m,n\in M$. The canonical homomorphism $L\to Aut(M)$, induced from the action of $L$ on $M$ is denoted by $\rho$. Thus, $\rho(x)(m)=\,^xm.$

Sometimes we use the additive notation. Instead of $^xm$, we write $xm$. Then the above identities have following form
$$\d(xm)=x+\d(m)-x$$
$$\d(n)m=n+m-n.$$

Let us denote $\im(\d)$ by $N$. Then $N$ is a normal subgroup of $L$. We set $A=\ker(\d)$ and $Q=L/N$. Thus one has an exact sequence
$$1\to A\to M\xto{\d} L\to Q\to 1.$$
It follows from the definition that $A$ is a central subgroup of $M$ and the action of $L$ on $M$ yields an action of $Q$ on $A$.

\begin{Ex*}
\begin{enumerate}
\item Let $A$ be a $Q$-module and $\d:A\to Q$ be the trivial homomorphism. Then $\d$ is a crossed module. 

\item For any abelian group $A$ one has a crossed module $A\to 0$.
% which we will denote by ${\bf B}(A).$

\item Let $N$ be a normal subgroup of a group $L$. Then the inclusion $i:N\to L$ is a crossed module, where the action of $L$ on $N$ is given by conjugation. 
%The two extreme cases $N=1$ and $N=L$ are especially important. 
%Starting  from a group $L$ in this way one obtains two crossed modules $1\to L$ and $L\xto{id} L$, denoted respectively by ${\bf O}(L)$ and ${\bf I}(L)$.

\item Let $M$ be a group. Denote by ${\bf A}(M)$ the crossed module $\iota:M\to Aut(M)$, where $\iota(m)$ is the inner automorphism $x\mapsto mxm\1$. The action of $Aut(M)$ on $M$ is the obvious one.

%\item Any homomorphism of abelian groups $\d:A\to B$ gives rise to a crossed module, where the action of $B$ on $A$ is trivial. Such crossed modules are called \emph{abelian crossed modules}. In particular, ${\bf B}(A)$ is abelian.
\end{enumerate}
\end{Ex*}

If $\d:M\to L$ and $\d':M'\to L'$ are crossed modules, a \emph{morphism of crossed modules} is a pair $(\al, \bb)$, where $\al:M\to M'$ and $\bb:L\to L'$ are group homomorphisms such that the diagram 
$$\xymatrix{M\ar[r]^\d\ar[d]_\al & L\ar[d]^\bb\\ M'\ar[r]_{\d'} &L'}$$
commutes and 
$$\al(^x m)=^{\bb (x)}\al(m).
$$
Here $m\in M$, $x\in L$.
Crossed modules and their morphisms form a category ${\bf Xmod}$. Observe that for any crossed module $\d:M\to L$ there is a morphism of crossed modules
$$\xymatrix{M\ar[r]^\d \ar[d]_{id} & L\ar[d]^\rho\\ M\ar[r]^{\iota} &Aut(M)
}$$

\subsection{Definition of cohomotopy}

Let $(M^\bu\xto{\d^\bu}L^\bu)$ be a cosimplicial crossed module, i.e. a cosimplicial object in the category ${\bf Xmod}$. We use additive notations for groups. We set
$$\pi^0(M^\bu\xto{\d^\bu}L^\bu)=\{m_0\in M^0| \d^0(m_0)=0, \dd^0(m_0)=\dd^1(m_0)\}$$
Here as usual $\dd^i$ denote the coface maps in a cosimplicial object. 

A \emph{$1$-cocycle} of $(M^\bu\xto{\d^\bu}L^\bu)$ is a pair $(m_1,x_0)$ such that $m_1\in M^1, x_0\in L^0$ and 
$$\dd^2m_1+\dd^0m_1=\dd^1m_1$$
$$\dd^0x_0=\d^1m_1+\dd^1x_0.$$
We let ${\sf Z^1} (M^\bu\xto{\d^\bu}L^\bu)$ be the set of all $1$-cocycles of $(M^\bu\xto{\d^\bu}L^\bu)$. If $(m_1,x_0)$ and $(n_1,y_0)$ are $1$-cocycles, then they are \emph{cohomotopic} provided there exists an element $m_0\in M^0$ such that
$$n_1=-\dd^1m_0+m_1+\dd^0m_0,$$
and
$$y_0=-\d^0 m_0+x_0.$$
One easily sees that cohomotopy is an equivalence relation and the set of equivalence classes is denoted by $\pi^1 (M^\bu\xto{\d^\bu}L^\bu)$.

A \emph{$2$-cocycle} of $(M^\bu\xto{\d^\bu}L^\bu)$ is a pair $(m_2,x_1)$ such that $m_2\in M^2, x_1\in L^1$ and
$$\dd^1 m_2+\dd^3m_2=\dd^2 m_2+^{\dd^2\dd^1x_1}\dd^0m_2$$
and
$$\dd^1x_1=\d^2 m_2+\dd^2x_1+\dd^0x_1.$$
We let ${\sf Z^2} (M^\bu\xto{\d^\bu}L^\bu)$ be the set of all $2$-cocycles of $(M^\bu\xto{\d^\bu}L^\bu)$. 

Let $(m_2,x_1)$ and $(n_2,y_1)$ be two $2$-cocycles. We write $(m_2,x_1) \sim (n_2,y_1)$ provided there exists $h^1\in M^1$ such that
$$m^2=-\dd^1h^1+n^2+\dd^2h^1+^{\dd^1x^1}\dd^0h^1$$
and
$$y^1=\d(h^1)+x^1.$$
The set ${\sf Z^2} (M^\bu\xto{\d^\bu}L^\bu)/\sim$ is denoted by ${\widehat{\pi}}^2(M^\bu\xto{\d^\bu}L^\bu)$ and is called \emph{thick $2$-cohomotopy}, compare with \cite[p.36]{dedecker}.

Let $(m_2,x_1)$ and $(n_2,y_1)$ be two $2$-cocycles. We call them \emph{cohomotopic} and write $(m_2,x_1) \simeq (n_2,y_1)$ provided there exist $h^1\in M^1$ and $z\in L^0$ such that
$$m^2=-\dd^1h^1+^{\d^1z}n^2+\dd^2h^1+^{\dd^1x^1}\dd^0h^1$$
and
$$y^1=-\dd^1z+\d(h^1)+x^1+\dd^0z.$$
One sets $$\pi^2(M^\bu\xto{\d^\bu}L^\bu):={\sf Z}^2(M^\bu\xto{\d^\bu}L^\bu)/\simeq.$$
It is clear that if $(m_2,x_1) \sim (n_2,y_1)$, then $ (m_2,x_1) \simeq (n_2,y_1)$. Hence one has a canonical surjection:
$${\widehat{\pi}}^2(M^\bu\xto{\d^\bu}L^\bu)\twoheadrightarrow \pi^2(M^\bu\xto{\d^\bu}L^\bu).$$

More on the topic of the cohomotopy of cosimplicial crossed modules will be explored in a follow-up paper \cite{mariam_preparation}.

\subsection{Cohomology with coefficients in crossed modules} \subsection{Definition}\label{eqcross} Let $G$ be a group and $\d:M\to L$ be a crossed module. We will say that $G$ \emph{acts} on $M\xto{\d} L$, or $M\xto{\d} L$ is a \emph{$G$-equivariant crossed module} provided $G$ acts on the groups $M$ and $L$ in such a way that the following identities hold:
$$\d(^x m)=\,^x\d(m),$$
$$^x(^\tau m) = ^{{^x}\tau}(^xm).$$
Here $m\in M$, $\tau\in L$ and $x\in G$.

\begin{Example*} Let $M$ be a $G$-group. We claim that ${\bf A}(M)=(M\xto{\iota} Aut(M))$ is a $G$-equivariant crossed module, where the action of $G$ on $Aut(A)$ is given by $$(^x\alpha)(m)=\,^x(\alpha(^{x^{-1}}m)).$$ 
To check this fact we first compute
$$(^x\iota(m))(n)= \, ^x(\iota(m)(^{x^{-1}}n))=\,^x(m\,(^{x^{-1}}n)m^{-1})=\, ^xm n\, ^x(m^{-1})=
\iota(^x m)(n)
$$
Secondly, we also have 
$$^{^x\alpha}(^xm)=\,^x(\alpha (^{x^{-1}}(^xm))= ^x(\alpha(m))=\,^x(^\alpha m)$$
So, both identities of Definition \ref{eqcross} hold.
\end{Example*}
 
%\subsection{Cosimplicial crossed module associated to a $G$-equivariant crossed module}
Let $\d$ be a crossed $G$-module. Then one obtains a cosimplicial crossed module $C^*(G, M\xto{\d} L)$, which in dimension $n$ is given by $$Maps(G^n, M)\to Maps(G^n, L).$$ The coface and codegenerecy maps are given by the same maps as in group cohomology. Then one puts 
$$H^i(G, M\xto{\d} L)=\pi^i(C^*(G, M\xto{\d} L)), \ \  i=0,1,2.$$

In this way one recovers the groups defined in \cite{borovoi}. We also set

$${\widehat H}^2(G, M\xto{\d} L)={\widehat \pi}^2(C^*(G, M\xto{\d} L).$$
%In particular, we see that 
%$$H^0(G,M\xto{\d} L)=\{m\in M| \d(m)=0 \quad {\rm and} \quad gm=m, \, g\in G\}.$$

%%%%%%%%%%%%%%%%%%%%%%%%%%%%%%%%%%%%%%%%%%%%%%%%%%%%%%%%%%%%%%%%%%%%
%%%%%%%%%%%%%%%%%%%%%%%%%%%%%%%%%%%%%%%%%%%%%%%%%%%%%%%%%%%%%%%%%%%%
%%%%%%%%%%%%%%%%%%%%%%%%%%%%%%%%%%%%%%%%%%%%%%%%%%%%%%%%%%%%%%%%%%%%
%%%%%%%%%%%%%%%%%%%%%%%%%%%%%%%%%%%%%%%%%%%%%%%%%%%%%%%%%%%%%%%%%%%%
%%%%%%%%%%%%%%%%%%%%%%%%%%%%%%%%%%%%%%%%%%%%%%%%%%%%%%%%%%%%%%%%%%%%
\section{Applications to the Schreier theory of group extensions} 
%%%%%%%%%%%%%%%%%%%%%%%%%%%%%%%%%%%%%%%%%%%%%%%%%%%%%%%%%%%%%%%%%%%%
%%%%%%%%%%%%%%%%%%%%%%%%%%%%%%%%%%%%%%%%%%%%%%%%%%%%%%%%%%%%%%%%%%%%
%%%%%%%%%%%%%%%%%%%%%%%%%%%%%%%%%%%%%%%%%%%%%%%%%%%%%%%%%%%%%%%%%%%%
%%%%%%%%%%%%%%%%%%%%%%%%%%%%%%%%%%%%%%%%%%%%%%%%%%%%%%%%%%%%%%%%%%%%
%%%%%%%%%%%%%%%%%%%%%%%%%%%%%%%%%%%%%%%%%%%%%%%%%%%%%%%%%%%%%%%%%%%%

In this section we recall the main results of the Schreier theory of extensions \cite{sch} and relate the results to the cohomotopy of cosimplicial crossed modules. Actually, we will present the result in a slightly more general form, where crossed modules are involved. The corresponding notion of an extension with respect to a given crossed module is due to Dedecker in the 60's.

\subsection{ ${\bf Ext}_{non-ab}(\Pi,M)$ and ${\bf wExt}_{non-ab}(\Pi,M)$}
Let $M$ and $\Pi$ be groups. An \emph{extension} of $\Pi$ by $M$ is a short exact sequence of groups:
\begin{equation}\label{B} 1\to M\to B\xto{\sigma} \Pi\to 1.\end{equation}
Following \cite[Section IV.8]{homology} we use additive notation for $M$ and $B$ and multiplicative notation for $\Pi$.
Two such extensions (\ref{B}) and
\begin{equation}\label{B'} 1\to M\to B'\xto{\sigma} \Pi\to 1\end{equation}
are \emph{congruent} provided they fit in a commutative diagram of group homomorphisms
 $$\xymatrix{1\ar[r] & M\ar[r]\ar[d]_{id} & B\ar[r]^{\sigma}\ar[d]^{\alpha} & \Pi\ar[r]\ar[d]^{id} &1\\
1\ar[r]&  M\ar[r] & B'\ar[r]^{\sigma'} &\Pi\ar[r] & 1}.$$
The set of congruency classes of extensions is denoted by ${\bf Ext}_{non-ab}(\Pi,M)$. Moreover, two extensions (\ref{B}) and (\ref{B'}) are called \emph{weakly congruent} if there exists an isomorphism $\alpha:B\to B'$ which fits in the commutative diagram
$$\xymatrix{1\ar[r] & M\ar[r]\ar[d]_{\tau} & B\ar[r]^{\sigma}\ar[d]^{\alpha} & \Pi\ar[r]\ar[d]^{id} &1\\
1\ar[r]&  M\ar[r] & B'\ar[r]^{\sigma'} &\Pi\ar[r] & 1}.$$
Denote by ${\bf wExt}_{non-ab}(\Pi,M)$ the weakly congruent classes of extensions. By definition, we have a surjective map
$${\bf Ext}_{non-ab}(\Pi,M) \to {\bf wExt}_{non-ab}(\Pi,M).$$

The main result of Schreier theory describes the set ${\bf Ext}_{non-ab}(\Pi,M)$ via the cohomotopy of cosimplicial crossed modules. 

We will obtain a more general result, which will describe the equivalence classes of $(M\xto{\d} L)$-extensions for a crossed module  $\d:M\to L$. To explain this translation we state the following basic fact about extensions.
\begin{Le} \label{i-iiiext} \begin{itemize}

\item [i)] For any extension (\ref{B}) the map $\zeta:B\to Aut(M)$ given by $\zeta(b)(m)=b+m-b$ is a group homomorphism and it fits in the following commutative diagram
$$\xymatrix{1\ar[r] & M\ar[r]\ar[d]_{\iota} & B\ar[r]^{\sigma}\ar[dl]^{\zeta} & \Pi\ar[r]&1\\
&  Aut(M)& && }$$
where as usual $\iota(m)$ is the inner  automorphism of $M$ defined by $m$.
\item [ii)] For any commutative diagram of group extensions 
$$\xymatrix{1\ar[r] & M\ar[r]\ar[d]_{id} & B\ar[r]^{\sigma}\ar[d]^{\alpha} & \Pi\ar[r]\ar[d]^{id} &1\\
1\ar[r]&  M\ar[r] & B'\ar[r]^{\sigma'} &\Pi\ar[r] & 1}$$
the following diagarm
$$\xymatrix{B\ar[dr]_{\zeta}\ar[rr]^\alpha & & B'\ar[dl]^{\zeta'}\\&Aut(M)&}
$$also commutes.

\item[iii)] For any commutative diagram of group extensions 
$$\xymatrix{1\ar[r] & M\ar[r]\ar[d]_{\tau} & B\ar[r]^{\sigma}\ar[d]^{\alpha} & \Pi\ar[r]\ar[d]^{id} &1\\
1\ar[r]&  M\ar[r] & B'\ar[r]^{\sigma'} &\Pi\ar[r] & 1}$$
and for any $b\in B$, $m\in M$ one has
$$\tau\1\zeta'(\alpha(b))\tau(m)=\zeta(b)(m).$$

\end{itemize}
\end{Le}
 
\begin{proof} i) is obvious and ii) is a particular case of iii) when $\tau=id$. For iii) observe that for any element $a\in M$ we have $\tau(m)=\alpha(m)$, hence we have
$\zeta'(\alpha(b))(\tau(m))=\alpha(b)+\tau(m)-\alpha(b)=\alpha(b+m-b)=\alpha(\zeta(b)(m))=\tau(\zeta(b)(m))$ and the result follows.

\end{proof}
%%%%%%%%%%%%%%%%%%%%%%%%%%%%%%%%%%%%%%%%%%%%%%%%%%%%%%%%%%%%%%%%
\subsection{$(M\xto{\d} L)$-\emph{extensions} and the sets ${\bf Ext}_{\d}(\Pi,M)$ and ${\bf wExt}_{\d}(\Pi,M)$} 
%%%%%%%%%%%%%%%%%%%%%%%%%%%%%%%%%%%%%%%%%%%%%%%%%%%%%%%%%%%%%%%%%%%%%%%%%%%%%%%

Lemma \ref{i-iiiext} suggests that we consider a more general situation. Let $\d:M\to L$ be a crossed module. Since $L$ acts on $M$, we have a canonical homomorphism $$\rho:L\to Aut(M)$$ given by $\rho(x)(m)=xm$.

\begin{De}\label{crext} An $(M\xto{\d} L)$-\emph{extension} of $\Pi$ by $M$ is a short exact sequence of groups (\ref{B}) together with a homomorphism  $\varrho:B\to L$ which fits in the commutative diagram
$$\xymatrix{1\ar[r] & M\ar[r]\ar[d]_{\d} & B\ar[d]_\zeta \ar[r]^{\sigma}\ar[dl]^{\varrho} & \Pi\ar[r]&1\\
&  L\ar[r]_\rho& Aut(M)&& }.$$
%such that $$\zeta=\rho\circ\varrho.$$
\end{De}
It follows that $b+m-b=\varrho(b)m$ for all $b\in B$ and $m\in M$. This definition goes back to Dedecker.

Two $(M\xto{\d} L)$-extensions (\ref{B}) and (\ref{B'})
%\begin{equation}\label{B'} 1\to M\to B\xto{\sigma} \Pi\to 1\end{equation}
are \emph{congruent} if there exists a homomorphism $\alpha:B\to B'$ (necessarily an isomorphism) such that the diagram
$$\xymatrix{ B \ar[dr]_\varrho \ar[rr]^\alpha& & B'\ar[dl]^{\varrho '}\\ &L}$$
commutes and $\alpha$ fits in the following commutative diagram of groups
$$\xymatrix{1\ar[r] & M\ar[r]\ar[d]_{id} & B\ar[r]^{\sigma}\ar[d]^{\alpha} & \Pi\ar[r]\ar[d]^{id} &1\\
1\ar[r]&  M\ar[r] & B'\ar[r]^{\sigma'} &\Pi\ar[r] & 1}.$$

The set of congruency classes of $(M\xto{\d} L)$-extensions is denoted by ${\bf Ext}_{\d}(\Pi,M)$. 

Observe  that for any group $M$ and the crossed module ${\bf A} (M)=(M\xto{\iota} Aut(M))$ one has

\begin{equation}\label{nonab=A}
{\bf Ext}_{non-ab}(\Pi,M)={\bf Ext}_{{\bf A}(M)}(\Pi,M).
\end{equation}

This follows from Lemma \ref{i-iiiext} and the fact that in this case $\rho$ is the identity map and hence $\varrho=\zeta$ is uniquely determined.

%%%%%%%%%%%%%%%%%%%%%%%%%%%%%%%%%%%%%%%%%%%%%%%%%%%%%%%%%%%%%%%%%%%%%%%%%%%%%%%%%%%%%%%
%\subsection{${\bf wExt}_{\d}(\Pi,M)$} We start with the following definition.
%%%%%%%%%%%%%%%%%%%%%%%%%%%%%%%%%%%%%%%%%%%%%%%%%%%%%%%%%%%%%%%%%%%%%%%%%%%%%%%%%%
\begin{De}\label{wextcr}
Two $(M\xto{\d} L)$-extensions (\ref{B}) and (\ref{B'}) are called \emph{weakly congruent} if there exists an isomorphism $\alpha:B\to B'$ and an element $\tau\in L$ such that the diagram  
$$\xymatrix{1\ar[r] & M\ar[r]\ar[d]_{\rho(\tau)} & B\ar[r]^{\sigma}\ar[d]^{\alpha} & \Pi\ar[r]\ar[d]^{id} &1\\
1\ar[r]&  M\ar[r] & B'\ar[r]^{\sigma'} &\Pi\ar[r] & 1}$$
commutes, and if for any $b\in B$ one has
\begin{equation}\label{tauverrho}
\tau^{-1} \varrho'( \alpha(b)) \tau=\varrho(b).
\end{equation}
\end{De}

Denote by ${\bf wExt}_{\d}(\Pi,M)$ the weak congruence classes of $(M\xto{\d} L)$-extensions. By definition, we have a surjective map
$${\bf Ext}_{\d}(\Pi,M) \to {\bf wExt}_{\d}(\Pi,M).$$

\begin{Le} For any group $M$ one has
$${\bf wExt}_{non-ab}(\Pi,M)={\bf wExt}_{{\bf A}(M)}(\Pi,M),$$
where ${\bf A}(M)=(M\xto{\iota} Aut(M))$ is the canonical crossed module associated to $M$. 
\end{Le}
\begin{proof} We already mentioned that the notion of an $(M\xto{\iota} Aut(M))$-extension is equivalent to the notion of an extension, see (\ref{nonab=A}). Hence we have to show that weak congruences in both frameworks are the same, but this is a consequence of part iii) of Lemma \ref{i-iiiext}.
%Assume (\ref{B}) and (\ref{B'})  are weakly congruent extensions. That is one has an isomorphism $\alpha$ such that $\sigma=\sigma'\circ \alpha$. Then the restriction $\tau$ of $\alpha$ on $M$ is an automorphism of $M$. Since $\rho$ is the identity map, we only need to check the equality (\ref{tauverrho}), or equivalently
%$$\varrho'(\alpha(b))(\tau m)=\tau(\rho(b)(m))$$
%In fact we have
%\begin{align*}
%\varrho'(\alpha(b))(\tau m)&= \alpha(b)+\tau(m)-\alpha(b)\\
%& =\alpha(b)+\alpha(m)-\alpha(b)\\
%&= \alpha(\rho(b)(m))\\
%&=\tau(\rho(b)(m)).
%\end{align*}
%Hence the extensions (\ref{B}) and (\ref{B'}) considered as  $(M\xto{\iota} Aut(M))$-extensions are also weakly congruence.
%Conversely, if the extensions (\ref{B}) and (\ref{B'}) considered as  $(M\xto{\iota} Aut(M))$-extensions are weakly congruences they are also weakly congruences as extensions, thank to commutative diagram in the Definition \ref{wextcr} and the fact, that $\rho$ is the identity for $A(M)$.
\end{proof}

%%%%%%%%%%%%%%%%%%%%%%%%%%%%%%%%%%%%%%%%%%%%%%%%%%%%%%%%%%%%%%%%%%%%%%%%
\subsection{Cocycle description of ${\bf Ext}_{\d}(\Pi,M)$ and ${\bf wExt}_{\d}(\Pi,M)$}\label{23.16.08.20}
%%%%%%%%%%%%%%%%%%%%%%%%%%%%%%%%%%%%%%%%%%%%%%%%%%%%%%%%%%%%%%%%%%%
The aim is to describe the  sets ${\bf Ext}_{\d}(\Pi,M)$ and ${\bf wExt}_{\d}(\Pi,M)$ via cocycles. 

For a given $(M\xto{\d} L)$-extension (\ref{B}) we choose a set theoretical section
$u:\Pi\to B$. We will always assume that $u(1)=0$. We set $\phi=\varrho\circ u:\Pi\to L$. 
\begin{Le} For any $x\in \Pi$ and $m\in M$ one has
\begin{equation}\label{h-is-gan} u(x)+m=\phi(x)m+u(x).
\end{equation}
\end{Le}
\begin{proof} In fact, we have
$ u(x)+m-u(x)=(\zeta u(x))(m)=\rho\varrho(u(x))(m)=\rho(\phi(x))m$ and the result follows.
\end{proof}
%This in fact defines an element $\phi(x)\in Aut(G)$. 
Define the map $f:\Pi\times \Pi\to M$ by
$$u(xy)=f(xy)+u(x)+u(y).$$
%The following is a classical result due to Schreir, see also \cite[Section IV.8]{homology}.

\begin{Le}\label{schreier1} The maps $\phi:\Pi\to L$ and  $f:\Pi\times \Pi\to M$ satisfy the following identities
\begin{equation}\label{phi-id}\phi(xy)=\d (f(x,y)) \phi(x)\phi(y),\end{equation}
\begin{equation}\label{f-2-cyc} f(xy,z)+f(x,y)=f(x,yz)+\phi(x)f(y,z)\end{equation}
%Here $\iota_a$ denotes  the inner automorphism $G\to G$, given by $x\mapsto axa^{-1}$.
\end{Le}
\begin{proof} We have
$$\phi(xy)=\varrho(u(xy))=\varrho(f(xy)+u(x)+u(y))=\varrho(f(x,y))\phi(x)\phi(y)$$
Since $\varrho\circ \kappa=\d$, we can write $\varrho(f(x,y))=\d(f(x,y))$ and (\ref{phi-id}) follows.

%For any $x,y\in \Pi$ and $g\in G$ we have
%$$u(xy)+g=\phi(xy)(g)+u(xy)=\phi(xy)(g)+f(x,y)+u(x)+u(y).$$
%We also have
%$$u(xy)+g=f(x,y)+u(x)+u(y)+g=f(x,y)+\phi(x(\phi(y)(g))+u(x)+u(y)$$
%Comparing two computations, we see that
%$$\phi(xy)(g)=f(x,y)+\phi(x\phi(y)(g))-f(x,y).$$
%This proves the equlaity (\ref{phi-id}).
For (\ref{f-2-cyc}) we compute $u(xyz)$ in two different ways:
$$u(x(yz))=f(x,yz)+u(x)+u(yz)=f(x,yz)+u(x)+f(y,z)+u(y)+u(z)=$$
$$=f(x,yz)+\phi(x)(f(y,z))+u(x)+u(y)+u(z)$$
and
$$u((xy)z)=f(xy,z)+u(xy)+u(z)=f(xy,z)+f(x,y)+u(x)+u(y)+u(z).$$
Comparing these computations, we obtain the equality (\ref{f-2-cyc}).
\end{proof}

\begin{Le}  Let $\phi:\Pi\to L$ and  $f:\Pi\times \Pi\to M$ be maps satisfying the conditions (\ref{phi-id}) and (\ref{f-2-cyc}). Then we have an $(M\xto{\d} L)$-extension (\ref{B}), where $B$ as a set is the cartesian product $B=M\times \Pi$, while the group  operation is defined by
$$(a, x)+(b,y)=(a+\phi(x)b-f(x,y),xy), \quad x,y\in \Pi, a, b\in M$$
and the structural maps are given by
 $$\kappa(a)=(a,1), \sigma(a,x)=x, \varrho(a,x)=\d(a) \phi(x).$$
\end{Le}
\begin{proof} The fact that $B$ is indeed a group is a result of Schreier, see \cite{homology}. The only non-trivial task is to show that $\varrho$ is a homomorphism. This requires checkingk the following equality
$$\d (a)\phi(x)\d(b)\phi(y)=\d(a+\phi(x)b-f(x,y))\phi(xy).$$
Replacing $\phi(xy)$ by $\d(f(x,y))\phi(x)\phi(y)$ and using the fact that $\d$ is a homomorphism, the last equality is equivalent to
$$\d(a)\phi(x)\d(b)=\d(a)\d(\phi(x)b)\phi(x).$$
Now recall that $\d(za)=z\d(a)z\1$ for any $z\in L, a\in M$. Thus
$$\d(a)\d(\phi(x)b)\phi(x)=\d(a)\phi(x)\d(b)\phi(x)\1\phi(x)=\d(a)\phi(x)\d(b)$$
and the proof is finished.
\end{proof}

%Our next observation is to interpret the pairs $(f,\phi)$ satisfying the conditions \ref{phi-id} and \ref{f-2-cyc} as $2$-cocycles in a cosimplicial crossed module.  In order to describe the corresponding cosimplicial crossed module, we need some notations. For the groups $\Pi,S$, denote by $C^*(\Pi,S)$ the following cosimplicial group
%$$C^n(\Pi,S)=Maps(\Pi^n,S)$$ and
%$$\dd^i(f)(x_1,\cdots,x_{n+1})=\begin{cases}
%f(x_2,\cdots,x_{n+1}),&  i=0\\
%f(x_1,\cdots,x_ix_{i+1},\cdots, x_{n+1}),&  0<i<n+1,\\
%f(x_1,\cdots,x_n),& i=n+1.\end{cases}$$
%$$\sigma^if(x_1,\cdots,x_{n-1})=f(x_1,\cdots, x_{i-1},1,x_i,\cdots, x_{n -1})$$
%We now introduce the following cosimplicial crossed module
%$$C^*(\Pi,M\xto{\d} L):=(C^*(\Pi,M)\xto{\d^*}C^*(\Pi, L)).$$ 
%where $\d^*$ is indiced  by the canonical homomorphism $\iota:G\to Aut(G)$. 
We observe that Lemma \ref{schreier1} says that for the functions $\phi$ and $f$, one has $(f,\phi)\in {\mathsf Z}^2(\Pi,M\xto{\d} L)$.

Our next aim is to translate the meaning of congruence and weak congruence in terms of the corresponding cocycles. %coboundaries?

\begin{Pro}\label{schreier2} Let $u:\Pi\to B$ and $u':\Pi\to B'$ be set sections of $\sigma:B\to \Pi$ and $\sigma':B'\to \Pi$ in the $(M\xto{\d} L)$-extensions (\ref{B}) and (\ref{B'}). Moreover, let $(f,\phi)$ and $(f,\phi')$ be the corresponding functions defined in Lemma \ref{schreier1}. 

\begin{enumerate}
 \item The extensions (\ref{B}) and (\ref{B'}) are congruent iff there exists a function $t:\Pi\to M$ such that
$$f'(x,y)=-t(xy)+f(x,y)+t(x)+\phi(x)(t(y))$$
and
$$\phi(x)=\d (t(x)) \phi(x)'.$$
 
 \item The extensions (\ref{B}) and (\ref{B'}) are weakly congruent iff there exists a function $t:\Pi\to M$ and an element $\tau\in L$ such that
$$f'(x,y)=-t(xy)+\tau f(x,y)+t(x)+\phi'(x)t(y)$$
and
$$\phi(x)=\tau\1\d t(x)\phi'(x) \tau.$$
\end{enumerate}
\end{Pro}

\begin{proof} We will only prove part (ii). The first part is quite similar and can be achieved by assuming $\tau=id$ in the following computations.

Assume (\ref{B}) and (\ref{B'}) are weakly congruent. Thus there exists an isomorphism $\alpha:B\to B'$ and an element $\tau\in L$ for which the conditions of Definition \ref{wextcr} hold. Since $\sigma'\circ \alpha=\sigma$, it follows that there exists a unique function $t:\Pi\to M$ for which 
$$\alpha(u(x)) = t(x)+u'(x).$$
We recall that
$$u(xy)=f(x,y)+u(x)+u(y).$$
Apply $\alpha$ to obtain
$$t(xy)+u'(xy)=\tau f(x,y)+t(x)+u'(x)+t(y)+u'(y).$$
This can be rewritten,
$$t(xy)+f'(x,y)+u'(x)+u'(y)=\tau f(x,y)+t(x)+\phi'(x)t(y)+u'(x)+u'(y)$$
giving us
$$f'(x,y)=-t(xy)+\tau f(x,y)+t(x)+\phi'(x)t(y).$$
This is the first equality. To show the second one, recall that $\phi(x)=\varrho(u(x))$ and $\phi'(x)=\varrho'(u'(x))$.
Hence by (\ref{tauverrho}) we have
\begin{align*}
\phi(x)&=\varrho(u(x))\\ &=
\tau^{-1} \varrho'( \alpha(u(x))) \tau\\ &=\tau^{-1} \varrho'(t(x)+u'(x))\tau\\& =\tau^{-1}\d(t(x))\phi'(x)\tau
\end{align*}

Conversely, assume there exists a $t:\Pi\to G$ for which both equations hold. We have to show that the extensions  (\ref{B}) and (\ref{B'}) are congruent. Define the map $\alpha:B\to B'$ by $\alpha(m+u(x))=\tau m+t(x)+u'(x)$. Our first claim is that $\alpha$ is a homomorphism. Since 
$$(m+u(x))+(n+u(y))=(m+\phi(x)n-f(x,y)) +u(xy)$$
our claim is equivalent to the following equality
$$\tau m+t(x)+u'(x)+\tau n +t(y) +u'(y)=\tau m +\tau \phi(x) n -\tau f(x,y) +t(xy)+u'(xy).$$
After cancelling bu $\tau m$, the LHS equals to
$$t(x) +\phi'(x) \tau n +\phi'(x) t(y) -f'(x,y)+u'(xy)$$
Replacing $f'(x,y)$ by the quantity $-t(xy)+\tau f(x,y)+t(x)+\phi'(x)t(y)$
and cancelling $u'(xy)$, we see that the claim is equivalent to
$$t(x)+\phi'(x)\tau n -t(x)=\tau \phi(x) n $$
which easily follows from our assertion that $\tau \phi(x)=\d t(x)\phi'(x) \tau$ and from the fact that
$\d z m=z+m-z$.
%\tau\phi(x)n= t(x)+\phi'(x)\tau n -t(x)

Thus $\alpha:B\to B'$ is a homomorphism. Since  $\alpha(m+u(x))=\tau m+t(x)+u'(x)$, by taking $x=1$, we see that the restriction of $\alpha$ to $M$ is the function induced by the action of $\tau$. Thus, the diagram in Definition \ref{wextcr} commutes and as a consequence $\alpha$ is an isomorphism. 

Thus, the only thing left to show is the validity of the equality (\ref{tauverrho}). We take $b\in B$, which is of the form $b=m+u(x)$ for uniquely defined $m\in M$ and $x\in \Pi$. Then we have
\begin{align*}
\varrho'(\alpha(b))&=\varrho(\tau m+t(x)+v'(x))\\&=\d(\tau m)\d(t(x))\phi'(x).
\end{align*}
By assumption, $\d t(x)\phi'(x)=\tau \phi(x)\tau\1$. Hence,
\begin{align*}
\varrho'(\alpha(b))&=\d(\tau m)\tau \phi(x)\tau\1\\
&=\tau \d(m)\tau\1\tau \phi(x)\tau\1\\
&=\tau\d(m)\phi(x)\tau\1 \\
&=\tau \varrho (b)
\end{align*}

\end{proof}
\begin{Co} There are bijections
$${\bf Ext}_{\d}(\Pi,M)\cong  {\widehat{\pi}}^2(C^*(\Pi,M\xto{\d} L)).$$
and
$${\bf wExt}_{\d}(\Pi,M)\cong  \pi^2(C^*(\Pi,M\xto{\d} L)).$$
\end{Co}

\begin{Co}\label{13.19.08.2020} Assume $\d:M\to L$ is a trivial homomorphism (so $M$ is an $L$-module). Then
$${\bf Ext}_{\d}(\Pi,M)\cong  \coprod_{\phi}H^2(\Pi,M_\phi)$$
where $\phi$ runs through the set of all homomorphisms $\Pi\to L$ and $M_\phi$ is a $\Pi$-module, which is $M$ as an abelian group and the action of $\Pi$ on $M$ is given by $x\cdot m:=\, ^{\phi(x)}m$.
\end{Co}
\begin{proof}
Since $\d$ is trivial, the conditions on the pair $(\phi,f)$ in \ref{phi-id}, \ref{f-2-cyc} say that $\phi$ is a group homomorphism (hence $M_\phi$ is a well defined $\Pi$-module) and $f$ is a $2$-cocycle in $C^*(\Pi, M_\phi)$. Hence the result. 
\end{proof}

%%%%%%%%%%%%%%%%%%%%%%%%%%%%%%%%%%%%%%%%%%%%%%%%%%%%%%%%%%%%%%%%%%
%%%%%%%%%%%%%%%%%%%%%%%%%%%%%%%%%%%%%%%%%%%%%%%%%%%%%%%%%%%%%%%%%%%
%%%%%%%%%%%%%%%%%%%%%%%%%%%%%%%%%%%%%%%%%%%%%%%%%%%%%%%%%%%%%%%%%%
%%%%%%%%%%%%%%%%%%%%%%%%%%%%%%%%%%%%%%%%%%%%%%%%%%%%%%%%%%%%%%%%%%%%
\section{Equivariant Bouquets}
%%%%%%%%%%%%%%%%%%%%%%%%%%%%%%%%%%%%%%%%%%%%%%%%%%%%%%%%%%%%%%%%%%
%%%%%%%%%%%%%%%%%%%%%%%%%%%%%%%%%%%%%%%%%%%%%%%%%%%%%%%%%%%%%%%%%%%
%%%%%%%%%%%%%%%%%%%%%%%%%%%%%%%%%%%%%%%%%%%%%%%%%%%%%%%%%%%%%%%%%%
%%%%%%%%%%%%%%%%%%%%%%%%%%%%%%%%%%%%%%%%%%%%%%%%%%%%%%%%%%%%%%%%%%

\subsection{Preliminaries on equivariant groupoids} All groupoids are small.

Let $G$ be a group. A \emph{$G$-groupoid} is a groupoid  object in the category of $G$-sets, or equivalently, it is a groupoid $\Gamma$ on which $G$ acts in the category of groupoids. More explicitly, for any $g\in G$ a functor $\Phi(g):\Gamma\to \Gamma$ is given. For an object $a$ of $\Gamma$, the value of $\Phi(g)$ on $a$ is denoted by $^ga$. Similarly, for a morphism $\alpha:a\to b$ of $\Gamma$, the value of the functor $\Phi(g)$ on $\alpha$ is denoted by $^g\alpha$. Thus, $^g\alpha:\,^ga\to \,^gb.$ By the definition of a functor we have
$$^g (id_a)=id_{\,^ga}, \quad {\rm and} \quad ^g(\alpha\circ \beta)=\, ^g\alpha\circ \,^ g\beta$$
for any composable morphisms $\alpha$ and $\beta$. 

One requires that $\Phi(1)$ is the identity functor and $\Phi(gh)=\Phi(g)\circ \Phi(h)$. Thus we have $$^{gh}a=\,^g(^ha) \quad {\rm and} \quad ^{gh}\alpha=\, ^g(^h\alpha).$$

Let $\Gamma$ and $\Gamma'$ be $G$-groupoids. A functor $F:\Gamma\to \Gamma'$ is called \emph{$G$-equivariant}, if
$$F(^g a)=\,^g\, F(a) \quad {\rm and} \quad F(^g\alpha)=\,^g\,F(\alpha).$$ 
Assume $F$ and $G$ are $G$-equivariant functors $\Gamma\to \Gamma'$. A natural isomorphism
$$\xi:F\Longrightarrow G$$
is called \emph{$G$-equivariant} if for any $g\in G$ and any object $a$ of $\Gamma$ one has 
$$\xi(^ga)=\, ^g\xi(a):\,^gF(a)\to \, ^gG(a).$$

A $G$-equivariant functor $F:\Gamma\to \Gamma'$ is called a \emph{$G$-equivalence} if there exists a $G$-equivariant functor $T:\Gamma'\to \Gamma$ and $G$-equivariant natural isomorphisms $FT\Longrightarrow Id_{\Gamma'}$ and $TF\Longrightarrow Id_{\Gamma}$. 

A $G$-equivariant functor $F:\Gamma\to \Gamma'$ is called a \emph{weak $G$-equivalence} if it is an equivalence of gropoids. It is well-known that this happens if and only if  $F$ is essentially surjective, full and faithful.

It is clear that any $G$-equivalence is a weak $G$-equivalence. 

%Call a $G$-groupoid $\Gamma$ \emph{cofibrant} if $G$ acts freely on the set of objects of $\Gamma$. 
Recall that a $G$-set $S$ is free if there exists a subset $X\subset S$ such that the map $G\times X \to S$ given by $(g,x)\mapsto \,^gx$ is a bijection. Such an $X$ is called a \emph{basis} of $S$.

\begin{Pro}\label{cofibrant} For any $G$-groupoid $\Gamma$ there exists a $G$-equivariant weak equivalence $\Gamma^c\to \Gamma$ such that the action of $G$ on the set of objects of $\Gamma^c$ is free.

%ii) If  $\Gamma'$ is cofibrant, then any $G$-equivariant weak equivalence  $F:\Gamma\to \Gamma'$ is a $G$-equivalence. 
\end{Pro}
\begin{proof}  Choose a free $G$-set $S$ and a surjective $G$-map $\pi$ from $S$ to the set of objects of $\Gamma$. Define a $G$-groupoid $\Gamma^c$ as follows. The set of objects of $\Gamma^c$ is $X$. For $x,y\in S$ we set 
$$Hom_{\Gamma'}(x,y):=Hom_{\Gamma}(\pi(x),\pi(y)).$$
The composition is induced by the composition law in $\Gamma$.
If $\alpha:x\to y$ is a morphism, then the same $\alpha$ is a morphism $\pi(x)\to \pi(y)$. To distinguish them, we denote the second morphism by $\pi(\alpha)$. In this way we obtain a functor $\pi:\Gamma'\to \Gamma$, which is obviously an equivalence of categories.

The action of $G$ on the set of objects of $\Gamma'$ is the same as on $X$. We extend this action on morphisms as follows. Assume $\alpha:x\to y$ is a morphism in $\Gamma'$ and $g\in G$. Then $^g\alpha:\, ^gx\to \,^gy$ is the unique morphism such that
$$\pi(\,^g\alpha)=\, ^g\pi(\alpha).$$
In this way one obtains an action of $G$ on $\Gamma^c$. By our construction $\pi$ is $G$-equivariant and this finishes the proof.

 \end{proof}
 
  \subsection{Bouquets}
 
 A $G$-groupoid $\Gamma$ is called a $G$-\emph{bouquet} \cite{duskin} if it is nonempty and connected. 
 
It is well-known that when $G$ is trivial, any bouquet $\Gamma$ is equivalent to a one object category corresponding to a group $A$, where $A$ is the group of authomorphisms of an object of $\Gamma.$  

We would like to investigate to what extent this result is true for arbitrary $G$. We are interested in classifying $G$-groupoids up to $G$-equivariant weak equivalences.

We have the following easy but important fact.

\begin{Le} \label{uzravi} Let $\Gamma$ be a $G$-bouquet, such that the action of $G$ on the set of objects of $\Gamma$ has a fixed object $x$. Then $A=Aut_\Gamma(x)$ is a $G$-group, and the full subcategory corresponding to $x$ is a $G$-equivariant subcategory of $\Gamma$, denoted by $A$. Moreover, the inclusion $A\to \Gamma$ is a $G$-equivariant weak equivalence.
\end{Le}

\begin{proof} Take $\alpha\in A$ and $g\in G$. So $\alpha:x\to x$ is a morphism. Hence $^g\alpha:\,^gc\to \,^gx$ is also a morphism, but $^gx=x$. Hence $^g\alpha\in A$. In this way we obtain the action of $G$ on $A$. Hence $A$ considered as a full subcategory of $\Gamma$ corresponding to $x$ is a $G$-equivariant subcategory. Since $\Gamma$ is connected, the inclusion is a weak equivalence and we are done.
\end{proof}

\subsection{Functions $\vartheta_{g,\lambda}$}

Unfortunately, the condition of Lemma \ref{uzravi} does not hold in many cases. By Lemma \ref{cofibrant} we can and will assume that the action of $G$ on the set of a $G$-bouquet $\Gamma$ is free. We will fix an object $x$ of $\Gamma$ and denote by $M$ the group of automorphisms of $x$ in $\Gamma$, so $M=Aut_{\Gamma}(x)$.

For any pair $(g,\lambda)$, where $g\in G$ and $\lambda:\, ^g x\to x$  is a morphism in $\Gamma$, we can consider the map $$\vartheta_{g,\lambda}:M\to M$$
given as follows. Take $M\ni \alpha:x\to x$ and consider the composite of morphisms
$$x\xto{\lambda^{-1}} \, ^gx\xto{\,^g\alpha} \,^gx\xto{\lambda } x$$
which is denoted by $\vartheta_{g,\lambda}(\alpha)$. Thus

\begin{equation}\label{1}
\vartheta_{g,\lambda}(\alpha)=\lambda\circ \, ^g\alpha\circ \lambda^{-1}.\end{equation}
Since $^g (id)=id$, we see that 
\begin{equation}\label{1'} \vartheta_{g,\lambda}(id)=id_A.\end{equation}

From the equality (\ref{1}) we obtain
\begin{equation}\label{2}
^g\alpha=\lambda ^{-1}\circ \vartheta_{g,\lambda}(\alpha)\circ \lambda.\end{equation}
If $\lambda'$ is also a morphism $^gx\to X$, then for $s=\lambda'\circ \lambda$ we have 
\begin{equation}\label{17}
\vartheta _{g,\lambda'}=\, ^{\d(s)}\vartheta_{g,\lambda}(\alpha).
\end{equation}

\begin{Le}\label{theta-g-lambda} The map $\vartheta_{g,\lambda}:M\to M$ is a group automorphism.
\end{Le}
\begin{proof} First we show that it is a group homomorphism. This requires to show that for any $\alpha,\beta\in M$ one has an equality
$$\vartheta_{g,\lambda}(\alpha\circ \beta)=\vartheta_{g,\lambda}(\alpha)\circ \vartheta_{g,\lambda}(\beta).$$
We have
 \begin{align*}
\vartheta_{g,\lambda}(\alpha)\circ \vartheta_{g,\lambda}(\beta) &= \lambda \circ \, ^g\alpha\circ \lambda^{-1}\circ \lambda\circ \, ^g\beta\circ \lambda^{-1}\\
&= \lambda\circ \, ^g\alpha\circ ^g\beta \circ \lambda^{-1}\\
&= \vartheta_{g,\lambda}(\alpha\circ \beta).
\end{align*}
The next aim is to show that, in fact, $\vartheta_{g,\lambda}$ is an automorphism of $M$. To this end, let us look at a morphism  $\lambda:\, ^gx\to x$ and act on it by $g^{-1}$ to obtain the morphism $^{g^{-1}}\lambda:x\to\, ^{g^{-1}}x$. We also have a morphism $\mu:\, ^{g^{-1}}x\to x$. Hence the composite $ \mu \circ \, ^{g^{-1}} \lambda:x\to x$ is an element of $M$, which is denoted by $\epsilon$. 
We claim that the composite map 
$$M\xto{\vartheta_{g,\lambda} }M\xto{\vartheta_{g^{-1},\mu}}M$$
equals to $\iota_\epsilon$, where $\iota_\epsilon :M\to M$ is the inner automorphism 
$$\iota_\epsilon(x)=\epsilon x \epsilon^{-1}.$$
In fact, for any $\alpha\in M$, we have 
$$\vartheta_{g^{-1}, \mu} \vartheta_{g, \lambda} (\alpha)=\mu\circ \, ^{g^{-1}}(\lambda \circ \,^g \alpha \circ \lambda^{-1})\circ \mu^{-1} =\epsilon\alpha\epsilon ^{-1} .$$

Now, we are in the position to show that the map $\vartheta_{g,\lambda}$ is bijective. 
By Claim the map $\vartheta_{g,\lambda}$ is surjective and $\vartheta_{g^{-1}, \mu}$  is injective. Replacing $g$ by $g^{-1}$ we see that  $\vartheta_{g,\lambda}$ is bijective.
%Question: what claim?
\end{proof} 

\subsection{Group extensions and equivariant bouquets} 

Let $\Gamma$ be a $G$-bouquet. We fix an object $x\in \Gamma$. This is possible because $\Gamma$ is non-empty.
We set $M=Aut_\Gamma(x)$. 

Denote by $G^x_\Gamma$, or simply $G_\Gamma$, the collection of all pairs $(\lambda, g)$, where $g\in G$ and $\lambda:\,^gx\to x$ is a morphism in $\Gamma$. Since $\Gamma$ is connected, the map $p:G_\Gamma\to \Gamma$ given by $p(\lambda,g)=g$ is surjective. 

We now define a binary operation $\bullet$ on $G_\Gamma$ by
$$(\lambda ,g)\bullet (\mu,h)=(\lambda \ ^g\mu, gh).$$
In other words, if $\lambda:\, ^gx\to x$ and $\mu:\, ^hx\to x$ are morphisms, then we consider the composite map
$$^{gh}x\xto{^g\mu}\, ^gx\xto{\lambda} x.
$$

\begin{Le} The operation $\bullet$ defines a group structure on $G_\Gamma$. Moreover, we have a group extension
\begin{equation}\label{canext}
0\to M\xto{\kappa} G_\Gamma^x\xto{p} G\to 0
\end{equation}
where $\kappa(m)=(m,1)$.
\end{Le}

\begin{proof}
We will only prove the associativity of $\bullet$. The rest is obvious.
We have
$$((\lambda_1 ,g_1)\bullet (\lambda_2 ,g_2))\bullet (\lambda_3 ,g_3)=(\lambda_1\, ^{g_1}\lambda_2, g_1g_2)\bullet (\lambda_3 ,g_3)=(\lambda_1\, ^{g_1}\lambda_2\, ^{g_1g_2}\lambda_3,g_1g_2g_3).
$$
On the other hand,
$$(\lambda_1 ,g_1)\bullet ((\lambda_2 ,g_2)\bullet (\lambda_3 ,g_3))=(\lambda_1 ,g_1)\bullet (\lambda_2\,^{g_2}\lambda_3,g_2g_3)=(\lambda_1\, ^{g_1}\lambda_2\, ^{g_1g_2}\lambda_3,g_1g_2g_3)
$$
and the associativity property follows.
\end{proof}

Now, we describe the inverse process. Starting with a group extension
$$1\to M\xto{\kappa}B\xto{\pi}G\to 1$$
we construct a groupoid $\Gamma(B)$, or simply $\Gamma$, as follows. The set of objects of $\Gamma$ is $G$. A morphism from $f$ to $g$, for $f,g\in G$, is a triple $(f,b, g)$, where $b\in B$ is an element such that $\pi(b)=g\1f$. If $(g,c, h)$ is also a morphism, then the composite of morphisms
$$f\xto{(f,b,g)} g\xto{(g,c,h)} h$$
is the triple $(f,cb, h)$. This is well-defined, since
$\pi(cb)=\pi(c)\pi(b)=h\1 g g\1f=h\1f.$
The action of the group $G$ on $\Gamma(B)$ is defined as follows. If $h\in G$, then the action of $h$ on the set of objects of $B(\Gamma)$ is given by $(h,g)\mapsto hg$, while on morphisms it is given by
$$^h(f,b, g)=(hf,b,hg).$$

We now show that these constructions are mutually inverese to each other. 
In fact, having the $G$-bouquet $\Gamma(B)$, we can choose $x=1$ as an object of $\Gamma(B)$. By our construction, the morphisms $1\to 1$ in $\Gamma(B)$ are triples $(1,b,1)$, such that $\pi(b)=1$. Thus, we have a conical % Question
identification
$M=Aut_{\Gamma(B)}(1)$. Now, the elements of the group $B_{\Gamma(B)}$ are tuples $(g,(g,b,1))$ with the property $\pi(b)=g$. One easily sees that
$( g,(g,b,1)))\mapsto b$ defines an isomorphism of groups $B_{\Gamma(B)}\to B$, inducing equivalences  of corresponding extensions.

Quite similarly, starting with a $G$-bouquet $\Gamma$, we can consider the 
$G$-bouquet $\Gamma(B_\Gamma)$. The objects are elements of $G$. For $f,g\in G$, a morphism from $f$ to $g$ is a triple $(f,\lambda,g)$, where $\lambda:\, ^{g\1f}x\to x$ is a morphism of $\Gamma$. Then the functor
$$\Psi:\Gamma(B_\Gamma^x)\to \Gamma$$
which on objects is given by $\Psi(g)=\, ^gx$, $g\in G$ and on morphisms by
$$\Psi(f,\lambda,g)=\, ^g\lambda$$
is a $G$-equiavriant weak equivalence.

\subsection{Equivarint bouquets over crossed modules} \label{eqbiucr}
 Let $\d:M\to L$ be a crossed module. Unlike in the previous section, we use multiplicative notation. 
 \begin{De}\label{351} We will say that a $G$-equivariant bouquet $\Gamma$ is defined over $\d:M\to L$ if an object $a$ of $\Gamma$ is given, as well as an isomorphism $\eta:M\to Aut_\Gamma(a)$. Moreover, for any pair $(g,\lambda)$, where $g\in G$ and $\lambda:\, ^g a\to a$ is a morphism in $\Gamma$, an element $\theta_{g,\lambda} \in L$ is given, such that $\vartheta_{1,\eta(m)}=\d(m)$ and 
 $\vartheta_{g,\lambda}=\rho(\theta_{g,\lambda})$.
 \end{De}
In this case we say that $(\Gamma, a, \eta, \theta)$ is a $G$-equivariant bouquet defined over $\d:M\to L$.

 In what follows, we identify $M$ and $ Aut_\Gamma(a)$ via $\eta$. The last condition of Definition \ref{351} is equivalent to the identity   
 $$\lambda\circ \,^g\alpha= \,^{\theta_{g, \lambda }}\alpha\lambda, \  \alpha\in Aut_\Gamma(a).$$
 
Let $(\Gamma, a,\eta, \theta)$ and $(\Gamma', a', \eta', \theta')$ be $G$-equivariant bouquets defined over $\d:M\to L$. A $G$-equivariant weak equivalence $F:\Gamma\to \Gamma'$ is defined over \emph{$\d:M\to L$}  if $F(a)=a'$, $F(\eta(m))=\eta'(m)$ for all $m\in M$, and for any $\lambda:\,^ga\to a$ one has $\theta_{g,\lambda}=\theta'_{g,F(\lambda)}$.

The following fact is a direct corollary of Lemma \ref{theta-g-lambda}. 
\begin{Co} If $\Gamma$ is a bouquet, $a$ is an object of $\Gamma$ and $M=Aut_\Gamma(a)$, then $\Gamma$ is defined over $\iota: M\to Aut(M)$ in a unique way.
\end{Co}

Now we relate the Definitions \ref{crext} and \ref{351}. To this end, we consider the  extension (\ref{canext}) associated to a $G$-equivariant bouquet $\Gamma$ at $x\in Ob(\Gamma)$:
$$1\to M\to B_\Gamma\to G\to 1$$
where $B_\Gamma$ is the collection of all pairs $(g,\lambda:\,^gx\to x)$. Recall that a structure of a $(M\xto{\d} L)$-extension on it is given by the homomorphism $\varrho:B_\Gamma\to L$ satisfying certain conditions. By declaring $\theta_{g,\lambda}$ to be the value of $\varrho$ on $(g,\lambda)$, we have the following fact.

\begin{Le}\label{17a} Let $\d:M\to L$ be a crossed module and let $\Gamma$ be a $G$-equivariant bouquet, $x$ be an object of $\Gamma$ and $M=Aut_\Gamma(x)$. Then $\Gamma$ is defined over $(M\xto{\d} L)$ iff the group extension (\ref{canext}) is a $(M\xto{\d} L)$-extension. 
\end{Le} 

So we can translate the results about group extensions to equivariant bouquets. In particular,  we use the cocycle description of $(M\xto{\d} L)$-extensions to classfy $G$-equivariant bouquets over $(M\xto{\d} L)$. According to Section \ref{23.16.08.20}, in order to do so, we first need to choose a section of $\pi:B_\Gamma\to G$. In our circumstances, this means that for any $g\in G$ we choose a morphism $\lambda_g:\,^gx\to x$. We assume that $\lambda_1=id$. Then we define maps $\phi:G\to L$ and $f:G\times G\to M$ by setting
 %$$\phi(g)=\theta_{g,\lambda_g}, \ \  f(g,h)=\lambda_g\circ ^g\lambda_h\circ \lambda_{gh}^{-1}.$$
  \begin{equation}\label{phi-lambda}
  \phi(g)=\theta_{g,\lambda_g}= \lambda_g \, ^g\alpha \lambda_g\1
  \end{equation}
  \begin{equation}\label{f-lambda}
  \ \  f(g,h)=\lambda_{gh} \circ \, ^g\lambda_h\1\circ \lambda_g\1.\end{equation}
 The fact that $f(g,h)$ is in $M$ follows from the diagrams
$$x\xto{\lambda_{g}^{-1}}\, ^{g}x\xto{^g\lambda_h\1} \,^{gh}x\xto{\lambda_{gh}} x.$$

Now, the results of Section \ref{23.16.08.20} imply the following results.
\begin{Le}\label{phipsi} \begin{itemize}
\item [i)] For any $g,h\in G$ one has the following equality in $L$:
$$\phi(gh)=\d(f(g,h))\phi(g)\phi(h).$$

\item [ii)] For any $g,h,t\in G$ one has
$$f(gh,t)  f(g,h)= f(g,ht)\, ^{\phi(g)}f(h,t).
$$

\item [iii)] For any $g\in G$ we have
$$f(1,g)=1=f(g,1).
$$

%\begin{proof} i) Take $\alpha:x\to x$. By formula \ref{1} we have
%\begin{align*} \phi(g)\circ\phi(h)(\alpha)&=l_g\circ \,^g(l_h\circ \,^h\alpha\circ \l_h^{-1})\circ l_g^{-1}\\ &=
%l_g\circ \,^gl_h\circ ^{gh}\alpha (\,^gl_h)^{-1}\circ l_g^{-1}\\ &=
%\psi(g,h)\circ l_{gh}\,^{gh}\alpha\circ l_{gh}^{-1}\circ \psi(g,h)^{-1}\\ &=\psi(g,h)\circ \phi(gh)(\alpha)\circ \psi(g,h)^{-1}
%\end{align*}
%
%Hence
%$$\phi(g)\circ\phi(h)=\iota_{\psi(g,h)}\circ \phi(gh).$$
%
%ii) We have $$^{gh}l_t=\,^g(\,^hl_t).$$
%The LFS equals to $l_{gh}^{-1}\circ \phi(gh,t)\circ l_{ght}$.
%
%The RHS equals to 
%\begin{align*} ^g(l_h^{-1}\circ \phi(h,t)\circ l_{ht}) &=(\,^gl_h)^{-1}\circ \, ^g\phi(h,t)\circ \, ^gl_{ht} \\ &=
%(l_g^{-1}\phi(g,h)\circ l_{gh})^{-1} l_g^{-1}\circ (\phi(g)(\psi(g)) l_g\circ l_g^{-1}\circ \psi(g,ht)\circ l_{ght} \\ &=l_{gh}%^{-1}\psi(g,h)^{-1}\circ (\phi(g)(\psi(h,t))\l_g\circ l_g^{-1}\circ \psi(g,ht)\circ l_{ght}.
%\end{align*}
%
%By comparing LHS and RHS we can conclude
%$$\psi(gh,t)=\psi(g,h)^{-1}\circ (\phi(g)(\psi(h,t))\circ \psi(g,ht)$$
%and the result follows.
%
%iii)By the equality \ref{3} we have $\psi(g,1)=l_g\circ\, ^gl_1\circ l_{g}^{-1}=l_g\circ\, ^g\id_A \circ l_{g}^{-1}=\id_A.$
%\end{proof}
\item [iv)] Let $(\lambda_g)_{g\in G}:\,^gx\to x$ and $(\lambda'_g)_{g\in G}:\, ^gx\to x$ be two families of morphisms. Define the map $s:G\to M$ by
$$s(g)=\lambda'_g\circ \lambda_g^{-1}:x\to x.$$
Then we have 
$$\phi'(g)=\,^{\d(s(g))}\phi(g) 
$$
and 
$$f(g,h)=s(gh)\1f'(g,h) s(g) ^{\phi(g)}s(h).
$$
\end{itemize}
\end{Le}
%\begin{proof} By construction of the function $s$, we have $l'_g=s_g\circ l_g$. Hence for the function $\phi'$ we have
%\begin{align*} c\phi'(g)(\alpha)&=l'_g\circ \,^g\alpha \circ l_g^{-1'}\\
%&=s(g)\circ l_g\circ \, ^g\alpha \circ l_g^{-1}\circ s(g^{-1}).
%\end{align*}
%Hence
%$$\phi'(g)=\iota_{s(g)}\circ \phi(g)$$
%and the first identity  follows. Quite similarly, we have
%\begin{align*} 
%\psi'(g,h)&=l_g'\circ \, ^gl_h'\circ {l^{'}}_{gh}^{-1}\\
%&=(s(g)\circ l_g)\circ \, ^g(s(h)\circ l_h)\circ (s(gh)l_{gh})^{-1}\\
%&=s(g)\circ l_g\circ \, ^gs(h)\circ \, ^gl_h\circ l^{-1}_{gh}\circ s(gh)^{-1}
%\end{align*}
%
%By the equality \ref{3} we have $^g;_h\circ l^{-1}_{gh}=l_g^{-1}\circ \psi(g,h).$
%Thus the previos equalities can be rewritten 
%\begin{align*}
%\psi'(g,h)&=s(g)\circ l_g\circ \, ^gs(h)\circ \, ^gl_h\circ l^{-1}_{gh}\circ s(gh)^{-1}\\
%&= s(g)\circ l_g\, ^gs(h)\circ l_g^{-1}\circ \psi(g,h)\circ s(gh)^{-1}\\
%&=s(g)\circ \phi(g)(s(h)\circ \psi(g,h)\circ s(gh)^{-1}.
%\end{align*}
%This finishes the proof.
%\end{proof}

As usual denote by $Q$ the cokernel of the map $\d:M\to L$. The canonical map $L\to Q$ is denoted by $\pi$.

\begin{Co} The composite map $\xi=\pi\circ \phi:G\to L$ is a group homomorphism, which is independent of the choice of a family of morphisms
$(\lambda_g:\,^gx\to x)_{g\in G}$.
\end{Co}

\begin{proof} This  is a direct consequence of Lemma \ref{phipsi}. 
\end{proof}

\begin{Pro}\label{20} Let $G$ be a group, $\d:M\to L$ be a crossed module and let $\phi:G\to L$, $f:G\times G\to M$ be maps satisfying the conditions
\begin{itemize}%[label=(\alph*)]
\item $\phi(1)=\id_M$,

\item For any $g,h\in G$ one has $f(1,h)=f(g,1)=1$, 

\item For any  $g,h\in G$ one has $$\phi(gh)=\d(f(g,h))\phi(g) \phi(h).$$

\item For any $g,h,t\in G$ one has
$$f(gh,t)f(g,h)= f(g,ht) \,^{\phi(g)}f(h,t).
$$
\end{itemize}
\begin{enumerate}[label=(\roman*)]
\item There exists a well-defined $G$-bouquet $\Gamma(\phi,f)$ (or simply $\Gamma$), whose set of objects is $G$. For any $g,h\in G$ a morphism from $g$ to $h$ is a triple $(g,\alpha, h)$, where $\alpha\in M$. The composite
 $$g\xto{(g,\alpha,h)}\, h\xto{(h,\beta,t)} t$$
 is defined by
 $$(h,\beta,t)\circ (g,\alpha,h)=(g,\beta\alpha,t),$$
 and the identity morphism of an object $g$ is $(g,1,g)$. The action of $G$ on the objects of $\Gamma$ is given by 
 $$^gh=gh,$$
while on morphisms it is given by
$$^t(g,\alpha,h)=(tg, f(t,h)\phi(t)(\alpha)f(t,g)\1,th).$$ 
Then $ (\Gamma, 1, \eta,\theta)$ is a $G$-equivariant bouquet defined over $(M\xto{\d} L)$, where
$\eta:M\to Aut_\Gamma(1,1)$ is given by $\eta(\alpha)=(1,\alpha,1)$, while $\theta$ assigns to a morphism $(g,\lambda,1):g\to 1$ the element
$$\theta_{g,(g,\lambda,1)}:=\phi(g).$$

\item Let $(\phi', \psi')$ be another pair satisfying similar relations. Suppose there is a function $s:G\to M$, such that $s(1)=1$ and 
$$\phi(g)=\,^{\d(s(g))} \phi'(g), \quad f'(g,h)=s(gh)\1 f(g,h) s(g)\, ^{\phi(g)}s(h).$$ 
Then there exists an isomorphism of groupoids $$F: \Gamma(\phi,f)\to \Gamma(\phi',f')$$ which is the identity on objects, and on morphisms it is given by $$F(g,\alpha,h)=(g,s(h)^{-1}\alpha s(g),h).$$ Moreover, $F$ is a $G$-equivariant weak equivalence defined over $\d:M\to L$.

\item  Let $(\Gamma, x,\eta,\theta)$ be a $G$-equivariant bouquet defined over $M\to L$, $\lambda_g:\, ^xg\to g$ be a family of morphisms, $g\in G$ and $(\phi,f)\in Z^2(G, M\to L)$ be the corresponding functions described in Lemma \ref{phipsi}. Then there is a unique $G$-equivariant weack equivalence $T:\Gamma(\phi,f)\to \Gamma$  which is given on objects by $T(g)=\, ^gx$ and on morphisms by
$$T(g,\alpha,h)=\lambda_h^{-1}\alpha \lambda_g.$$
Moreover, $T$ is defined over $M\to L$.
\end{enumerate}
\end{Pro} 

\begin{proof}
\begin{enumerate}[label=(\roman*)]
\item The associativity of the composition law follows from the associativity of the multiplication in $A$. The unit axiom is obvious, hence $\Gamma$ is a small category. Since $(h,\alpha^{-1},g)$ is 
the inverse of $(g,\alpha,h)$, and for any $g\in G$ we have a morphism
$$\lambda_g=(g,1,1):g\to 1$$
we see that $\Gamma$ is a nonempty and connected groupoid. Hence it remains to show that $\Gamma$ is a $G$-groupoid. 
This means verifying the following identities:

\begin{equation}\label{5}
^s((h,\beta,t)\circ (g,\alpha,h))=\, ^s(h,\beta,t)\circ\, ^s (g,\alpha,h),
\end{equation}

\begin{equation}\label{6}  ^g(^h(s,\alpha,t))=\,^{gh}(s,\alpha,t).
\end{equation}

The first identity is a direct consequence of the fact that $\phi(s)$ is an automorphism of $A$. For the second identity, observe that we have
$$(s,\alpha, t)=(1,1,t)(1,\alpha,1)(s,1,1).$$
We also have $(1,1,t)=(t,1,1)^{-1}$. Hence, it follows from (\ref{5}) that it suffices to prove the identity (\ref{6}) in two particular cases: i) when $t=s=1$ and ii) when $t=\alpha=1$. 

For i) observe that we have
\begin{align*}
^g(^h(1,\alpha,1)) &= \, ^g(h,\phi(h)(\alpha),h)\\
&= (gh,f(g,h)\phi(g)\phi(h)(\alpha)\psi(g,h)\1,gh)
\end{align*}
By the third condition, the last expression is equal to $(gh, \psi(gh), gh)=\,^{gh}(1,\alpha,1)$. Hence, the case i) is done.  

For ii) we have
\begin{align*}
^g(^h(1,1,s))&= \, ^g(h, f(h,s),hs)\\
&= (gh,f(g,hs)^{\phi(g)}f(h,s)f(g,h)\1,ghs).
\end{align*}
by the fourth condition the last expression is equal to $(gh,f(gh,s),ghs)=\, ^{gh}(1,1,s)$. This finishes the proof.
%%Question: first, cond. c, then cond. iv. Something's off.
\item If $(g,\alpha,h)$ and $(h,\beta, t) $ are two composable morphisms in $\Gamma(\phi,f)$, then 
\begin{align*}F(h,\beta, t)\circ F (g,\alpha,h)&=(g, s(t)^{-1}\beta s(h) s(h)^{-1}\alpha s(g),t)\\
&=F\left ( (h,\beta, t)\circ (g,\alpha,h)\right ).\end{align*}
Thus, $F$ is indeed a functor. Since it is full and faithful and the identity on objects, it is an isomorphism of categories. To show that $F$ is $G$-equivariant, we need to show
$$F(^t(g,\alpha, h))=\, ^tF(g, \alpha,h).$$
Since $F$ is a functor, we need to check the equality in two particular cases, as was done in part (i): when $g=h=1$ and $g=\alpha=1$. In the first case we have:
$$^tF(1,\alpha,1)=\,^t(1,\alpha,1)=(t,\phi'(t)(\alpha),t)=(t,s(t)^{-1}\phi(t)(\alpha)s(t),t)=F(^t (1,\alpha,1)).$$
For the second case, we have
$$F(^t(1,1,h))=F(t, f(t,h),th)=(t,s(th)\1 f(t,h)s(t),th).$$
By assumption, this can be rewritten as
$$= (t,f'(t,h)^{\phi(t)}(s(h)\1),th)=\,^t(1,s(h)\1,h)=\,^tF(1,1,h).$$
Thus $F$ is $G$-equivariant. Comparing the definitions we see that it is indeed defined over $M\to L$.

\item From the description it is clear that $T$ is full and faithful. Since both groupoids are connected, $T$ is an equivalence of categories. Next, we will show that $T$ is $G$-equivariant. This is clear on objects. To check for morphisms it suffices to consider only the morphisms of the form $(1,\alpha,1) $ and $(1,1,h)$, see  the proof of part (i). Then we have
$$T(^t(1, 1,h))=T(t,f(t,h),th)= \lambda_{th}\1 f(t,h)\lambda _t=\, ^t(\lambda_h)\1=\, ^tT(1,1,h).$$
Here we used the equality (\ref{f-lambda}). For the second case we have
$$T(^t(1,\alpha,1))=T(t,\phi(t)(\alpha),t)=\lambda_t\1\phi(t)(\alpha)\lambda _t=\,^t\alpha =\, ^tT(1,\alpha,1).$$
Comparing the definitions we see that $T$ is
indeed defined over $M \to L$.

\end{enumerate}

\end{proof}
  
To summarise,  if  $(\Gamma, x,\eta, \theta)$ is a $G$-equivariant bouquet defined over $M\to L$, we can choose a family of morphisms $\lambda_g:\, ^gx\to x$, $g\in G$. Then  Lemma \ref{phipsi} gives an element $(\phi,f)\in Z^2(G, M\to L)$. The corresponding class in $H^2(G, M\to L)$ is independent of the chosen $\lambda_g$ (thanks to Lemma \ref{phipsi}, part iv). It is called \emph{the characteristic class} of $(\Gamma, x,\eta, \theta)$ and is denoted by $ch( \Gamma, x,\eta, \theta)$. By Proposition \ref{20}, part (i), any element of $H^2(G, M\to L)$ is a characteristic class of a  $G$-equivariant bouquet defined over $M\to L$. Moreover, if  $(\Gamma', x',\eta', \theta')$ is also a $G$-equivariant bouquet defined over $M\to L$ such that
	$$ch( \Gamma, x,\eta, \theta) = ch( \Gamma', x',\eta', \theta'),$$
then there exists a $G$-equivariant bouquet $(\Gamma'', x'',\eta'', \theta'')$ defined over $M\to L$ and $G$-equivariant weak equivalences defined over $M\to L$
$$(\Gamma, x,\eta, \theta)\longleftarrow (\Gamma'', x'',\eta'', \theta'')\longrightarrow
(\Gamma', x',\eta', \theta'').$$
In fact, we can choose families of morphisms $\lambda_g$ in $(\Gamma, x,\eta, \theta)$ and $\lambda_g'$ in $(\Gamma', x',\eta', \theta')$. Assume $(\phi,f)$ and $(\phi',f')$ are the corresponding cocycles. By part (iii) of the previous proposition, we have $G$-equivariant weak equivalences $\Gamma(\phi, f)\to \Gamma$ and $\Gamma'(\phi', f')\to \Gamma' $ defined over $M\to L$. 
By our assumptions $(\phi,f)$ and $(\phi',f')$ are homotopic, so we can use part (ii) of the previous proposition to construct a $G$-equivariant weak equivalence $\Gamma(\phi, f)\to \Gamma(\phi', f')$ defined over $M\to L$. So we can take $\Gamma''$ to be $\Gamma(\phi, f)$.

%%%%%%%%%%%%%%%%%%%%%%%%%%%%%%%%%%%%%%%%%%%%%%%%%%%%%%%%%%%%%%%%%%%%%%%%%%%%%%%%%%%%%%%%%%%%%%%%%%%%%%%%%%%%%%%%%%%%%%%%%%%%%%%%%%%%%%%%%%%%%%%%%%%%%%%%%%%%%%%%%%%%%%%%%%%%%%%%%%%%%%%%%%%%%%%%%%%%%%
%%%%%%%%%%%%%%%%%%%%%%%%%%%%%%%%%%%%%%%%%%%%%%%%%%%%%%%%%%%%%%%%%%%%%%%
%%%%%%%%%%%%%%%%%%%%%%%%%%%%%%%%%%%%%%%%%%%%%%%%%%%%%%%%%%%
%%%%%%%%%%%%%%%%%%%%%%%%%%%%%%%%%%%%%%%%%%%%%%%%%%%%%%%%%%%%%%%%%%%%%%%
%%%%%%%%%%%%%%%%%%%%%%%%%%%%%%%%%%%%%%%%%%%%%%%%%%%%%%%%%%%
\section{Bitorsors over crossed modules} 
\subsection{The first cohomology and bitorsors over crossed modules} 
%%%%%%%%%%%%%%%%%%%%%%%%%%%%%%%%%%%%%%%%%%%%%%%%%%%%%%%%%%%%%%%%%%%%%%%
%%%%%%%%%%%%%%%%%%%%%%%%%%%%%%%%%%%%%%%%%%%%%%%%%%%%%%%%%%%
%%%%%%%%%%%%%%%%%%%%%%%%%%%%%%%%%%%%%%%%%%%%%%%%%%%%%%%%%%%%%%%%%%%%%%%
%%%%%%%%%%%%%%%%%%%%%%%%%%%%%%%%%%%%%%%%%%%%%%%%%%%%%%%%%%

In this section we intoduce the notion of a bitorsor over a crossed module, generalising the classical notion of a bitorsor \cite{breen}, and then we relate the isomorphism classes of such bitorsors with the first cohomology of a group with coefficient in a crossed module.

We first recall the classical relationship between torsors and the first cohomology following \cite{serre}.  Let  $G$ be a group and $M$ be a left $G$-group. Assume $P$ is a set on which $M$ acts from the right and $G$ acts from the left in such a way that 
$$^x(pm)=\,^xp\, ^xm$$
holds for any $x\in G$, $p\in P$ and $m\in M$. Such a set $P$ is a \emph{ right $M$-torsor} if  it is nonempty and for any $p,q\in P$ there exists a unique $m_0\in M$, denoted by $[p\diagdown q]$, such that $q=pm_0$. Thus, we have

\begin{equation}\label{q=pm}
q=p[p\diagdown q].
\end{equation}

In a similar way one can define left torsors. 

It is well known that if $f:G\to M$ is a $1$-cocycle, that is $f(xy)=f(x)\, ^xf(y)$, then the corresponding torsor is $P$, where $P=M$ as a set, on which $M$ acts by right translations and a new left action of $G$ on $P$ is defined by
$$^{x\cdot}m=f(x)\, ^xm.$$
Conversely, if $P$ is a right $M$-torsor and an element $p_0\in P$ is chosen, then one can define the function $f:G\to M$ by
\begin{equation}\label{tor.1.coc}
f(x)=[p_0\diagdown \,^xp_0],
\end{equation}
or, equivalently, by $^xp_0=p_0f(x)$. Then $f$ is a $1$-cocycle, yielding a bijection between isoclasses of right $M$-torsors and $H^1(G,M).$  The torsor  corresponding to the trivial $1$-cocycle is simply $M$ on which $M$ acts by left translation with the given action of $G$. Obviously, the element $p_0=1\in M$ satisfies the property $^xp_0=p_0$. Conversely, if a torsor $P$ has a $G$-fixed element $p_0$, then the map $\tau:M\to P$ given by $\tau(m)=p_0m$ is an isomorphism of right $M$-torsors. Hence a $G$-fixed point in a torsor defines its trivialisation.

We wish to find a similar bijection for $H^1(G, (M\xto{\d} L))$. Recall that 
$$H^1(G, (M\xto{\d} L)):=Z^1(G,  (M\xto{\d} L))/\sim $$
where $Z^1(G,  (M\xto{\d} L))$ is the set of all $1$-cocycles of $G$ with coefficients in $M\xto{\d}L$, that is the set of all pairs  $(f,\tau)$, where  $f:G\to M$ is a map and $\tau$ is an element in $L$ such that
\begin{equation}\label{f=1coc} f(xy)=f(x) \, ^xf(y),
\end{equation}
\begin{equation}\label{detau}
\tau=\d f(x)\, ^x\tau.
\end{equation}
Next, if $(f,\tau), (f',\tau')\in Z^1(G,  (M\xto{\d} L))$, then we write $(f,\tau)\sim (f',\tau')$ if there exists $m\in M$ such that $\tau= \d(m)\tau'$ and $f(x)=m f'(x)\, (^xm)^{-1}$.

\begin{De}\label{def.bit.cr} Let $M\xto{\d} L$ be a $G$-crossed module. An $(M\xto{\d} L)$-bitorsor is a pair $(P,\alpha)$, where $P$ is a right $M$-torsor and $\alpha:P\to L$ is a function such that the following two conditions hold 
\begin{itemize}
\item [i)] $\alpha(\,^xp)=\, ^x\alpha(p)$, for all $x\in G$,

\item [ii)] $\alpha(pm)=\alpha(p)\d(m)$.
\end{itemize}

\end{De} 

\begin{Le}\label{zvlaal} For $p,p'\in P$ we have   
$$\alpha(p) \d([p\diagdown p'])=\alpha(p').$$
\end{Le}

\begin{proof} In fact, we have $p'=pm$, where $m=[p\diagdown p']$. Hence
$$\alpha(p')=\alpha(pm)=\alpha(p)\d (m).$$

\end{proof}

If $(P,\alpha)$ and $(P',\alpha')$ are $(M\xto{\d} L)$-bitorsors, then a \emph{morphism} $(P,\alpha)\to (P',\alpha')$ is a map $\phi:P\to P'$ such that $\phi$ respects the actions of $G$ and $M$ and $\alpha'\circ \phi=\alpha.$

Since any morphism of $M$-torsors is a bijective map, it follows that any morphism of $(M\xto{\d} L)$-bitorsors is an isomorphism. Denote by ${\bf Bitor}_G(M\xto{\d} L)$ the set of isomorphism classes of $(M\xto{\d} L)$-bitorsors.

\begin{Pro}\label{26.19.08.2020} Let $(M\xto{\d} L)$ be a crossed module. Then one has a bijection
$${\bf Bitor}_G(M\xto{\d} L) \cong H^1(G,(M\xto{\d} L)).$$
\end{Pro} 

\begin{proof} Take an $(M\xto{\d} L)$-bitorsor $P$ and choose an element $p_0\in P$. Since $P$ is a right $M$-torsor, it defines a function $f:G\to M$ by the equality (\ref{tor.1.coc}), and as already said, $f$ satisfies the condition (\ref{f=1coc}). We also set $$\tau=(\alpha(p_0))^{-1}.$$ By the conditions in Definition \ref{def.bit.cr} we have,
\begin{align*}
\d f(x)\, ^x\tau&=\d f(x)(\alpha(\,^xp_0))^{-1}\\ &=\d f(x)(\alpha(p_0f(x)))^{-1}\\ &= \d f(x)(\alpha(p_0)\d (f(x)))^{-1}\\ &=(\alpha(p_0))^{-1}\\  &=\tau.
\end{align*}

Hence, the condition (\ref{detau}) also holds. Thus $(f,\tau)\in Z^1(G,(M\xto{\d} L)).$ Let us show that the class of the pair $(f,\tau)$ in $H^1(G,(M\xto{\d} L))$ depends only on  the isomorphism class of $P$. In fact, if $\phi:P\to P'$ is an isomorphism of $(M\xto{\d} L)$-bitorsors and $p_0'\in P'$ is chosen arbitrarily, we have $p_0'=\phi(p_0)m$, where $m=[\phi(p_0)\diagdown p_0']$. Now we consider $^x(p_0')$ and compute it in two different ways. Firstly, we have
$$^x(p_0')=p_0'f'(x)=\phi(p_0)mf'(x)$$
Secondly, we also have
\begin{align*}
^x(p_0')&=\,^x(\phi(x)m)\\ &=\phi(^xp_0)\,^xm\\&=\phi(p_0f(x))\,^xm\\&=\phi(p_0)f(x)\, ^xm.
\end{align*}

Comparing these expressions we obtain 
$mf'(x)=f(x)\, ^xm.$

Next, we consider the action of $\alpha$ on $ p_0$. Since $P$ is a right $M$-torsor, we can write $p_0'=\phi(p_0)m$ where $m=[\phi(p_0)\diagdown p_0']$. So, we can write
\begin{align*}
(\tau')^{-1} &= \alpha'(p_0')\\ &=\alpha'(\phi(p_0)m)\\ & =\alpha(p_0)m\\ & =\tau^{-1}m.
\end{align*}

Hence $(f,\tau)\sim (f',\tau')$. Thus, we have constructed the well-defined map
$${\bf Bitor}_G(M\xto{\d} L) \to H^1(G,(M\xto{\d} L)).$$

To construct the inverse map, we proceed as follows. Take a pair $(f,\tau)\in Z^1(G, (M\xto{\d} L))$ and define $P$ to be a right $M$-torsor corresponding to the $1$-cocycle $f$. Thus $P=M$ as a set. The right action of $M$ on $P=M$ is given by the multiplication in $M$, while the action of $G$ on $M$ is given by $^{x\cdot} m=f(x)\, ^xm$. It remains to define the function $\alpha:P=M\to L$. We take 
$$\alpha(m)=\tau^{-1} \d(m).$$
We have to check that the conditions i)-ii) of the Definition   \ref{def.bit.cr} are fulfilled.

\begin{itemize}
	\item [i)] We have $$\alpha(^{x\cdot} m)=\alpha(f(x)\, ^xm)=\tau^{-1}\d (f(x)) \, ^x \d (m).$$ By (\ref{detau}) we have $\tau^{-1}=(\,^x\tau)^{-1}\d (f(x))^{-1}$. Hence we can rewrite
\begin{align*}
\alpha(^{x\cdot} m)& =(^x\tau)^{-1}\, ^x\d (m)\\ &=\, ^x(\tau^{-1}\d (m))\\ & =\, ^x\alpha(m)
\end{align*}
and the condition i) follows.
\item [ii)] We also have 
$$\alpha(m_1m)=\tau^{-1}\d (m_1m)=\tau^{-1} \d (m_1)\d (m)=\alpha(m_1)\d (m)
$$
and ii) follows.
\end{itemize}

Denote this $(M\xto{\d} L)$-bitorsor by $P_{(f,\tau)}$. Assume $(f,\tau)\sim (f',\tau')$. Thus $\tau= \d (m_0)\tau'$ and $f(x)=m_0 f'(x)\, (^xm_0)^{-1}$ for an element $m_0\in M$. Then the map $\phi:P\to P'$ is an isomorphsim, where $\phi(m)=m_0^{-1}m$. In this way we constructed the inverse map $ H^1(G,(M\xto{\d} L))\to {\bf Bitor}_G(M\xto{\d} L)$, finishing the proof.

%then $P_{(f,\tau)}$ and $P_{(f',\tau')}$ are isomorphic.
\end{proof}

We also have the following proposition.

\begin{Pro}\label{bitisweak} Let $(P,\alpha)$ be a $(M\xto{\d} L)$-bitorsor. For an element $m\in M$ and $p\in P$, we set
$$m*p:=p\, ^{\xi(p)}m$$
where  $\xi(p)=\alpha(p)^{-1}.$ 
With this operation $P$ is a left $M$-torsor and the left and right actions of $M$ on $P$ are related by the rule
$$(m*p)n=m*(pn).$$
\end{Pro}

\begin{proof} In fact, we have
$$m*(pn)=pn \, ^{\xi(pn)} m.$$
Since $\xi(pn)=(\alpha(pn))^{-1}=(\alpha(p)\d (n))^{-1}=\d (n^{-1}) \xi(p)$, we have 
$$m*(pn)=pn\,^{\d (n^{-1}) \xi(p)}m= pn n^{-1}\, ^\xi(p)mn=p\, ^{\xi(p)}_mn=(m*p)n.$$
This proves the last identity. Based on we will prove that %question: based on what?
$$m_1*(m_2*p)=(m_1m_2)*p$$
In fact, we have
\begin{align*}
m_1*(m_2*p)&=m_1*(p\, ^{\xi(p)}m_2)\\ &=(m_1*p)\, ^{\xi(p)}m_2\\ &=p\, ^{\xi(p)}m_1 ^{\xi(p)}m_2\\ &=p\, ^{\xi(p)}(m_1m_2)\\ &=(m_1m_2)*p.
\end{align*}

Since $1*p=p$, we see that $P$ is also a left $M$-set. Our next claim is that for any $x\in G$, $p\in P$ and $m\in M$ we have
$$^g(m*p)=(\, ^gm)*(\, ^gp).$$ 
In fact,
\begin{align*}
(\, ^gm)*(\, ^gp)&=(\, ^gp)\, ^{\xi(\,^gp)}(^gm)\\ &=(\, ^gp)\,^{^g\xi(p)}(^gm)\\ &=(\, ^gp)\,^g(^\xi(p)m)\\ &=\, ^g(p\,^{\xi(p)}m)\\ &=\, ^g(m*p).
\end{align*}

It remains to show that for any $p,q\in P$ there exists a unique $m\in M$ such that $q=mp$. In fact, for given $p$ and $q$, the equation $q=mp$ is equivalent to $q=p\,^{\xi(p)}m$. The last one is equivalent to $^{\xi(p)}m=[p\diagdown q]$, which obviously has a unique solution $m=\,^{\xi(p)^{-1}}([p\diagdown q])$.
\end{proof}

 {\bf Remark}. Proposition \ref{bitisweak} shows that any $(M\xto{\d} L)$-bitorsor is also an $M$-bitorsor in the classical sense  \cite{breen}. In fact, $P$ is an $M$-bitorsor in the classical sense iff it is an $(M\xto{\iota} Aut(M))$-bitorsor. In fact, if $P$ is an $M$-bitorsor, that is $M$ acts on $P$ on the left and right in a compatible way and $P$ is both a left and a right $M$-torsor, then we can define the map $\alpha: P\to Aut(M)$ by $\alpha(p)=\xi(p)^{-1}$, where $\xi(p)(m)$ is the unique element of $M$ for which $mp=p\xi(p)(m)$. A direct computation shows that $(P,\alpha)$ is an $(M\xto{\iota} Aut(M))$-bitorsor.

%\end{document}

%\begin{proof} By Lemma \ref{bitonecocyc} we have a well-defined map
%$$H^1(G,(M\xto{\d} L))\to {\bf Bitor}_G(M\xto{\d} L) $$
%\end{proof}

\subsection{An obstruction theory}

For a crossed module $(M\xto{\d} L)$, we set $$A:=Ker(\d) \quad {\rm and}\quad Q:=Coker(\d).$$
Then $A$ is a central subgroup of $M$, and the action of $L$ om $M$ induces the action of $Q$ on $A$. To summarize this situation, we will say that we have a crossed extension
$$0\to A\to M\xto{\d}\, L\xto{\pi}\, Q\to 0.$$

\begin{Le} Let  $P$ be a $(M\xto{\d} L)$-bitorsor. Choose an element $p\in P$. Then $\pi(\alpha (p))\in Q$ is independent of the chosen $p\in P$. 
\end{Le}

\begin{proof} Take two elements $p,q\in P$. Then by Lemma \ref{zvlaal} we have $\alpha(p)\d (m)= \pi(\xi(q)),$
where $m=[p\diagdown q]$. Hence the result.
\end{proof}

Thus $$\pi_*(P):=\pi(\xi(p))\in Q$$ is an invariant of the $(M\xto{\d} L)$-bitorsor $P$. The main property of this invariant is given by the following Lemma.

\begin{Le} Let $P$ be a $(M\xto{\d} L)$-bitorsor. Then 
$$\pi_*(P)\in H^0(G,Q).$$
\end{Le}

\begin{proof} In the proof  of Proposition \ref{26.19.08.2020} we have seen that for any $p_0\in P$ one has 
the equality (\ref{detau})  for $\tau\1=\alpha(p_0)$, which implies the result. 
\end{proof}

{\bf Question}. Let $(M\xto{\d} L)$ be a crossed module and  $a\in H^0(G,Q)$. Does there exist a $(M\xto{\d} L)$-bitorsor $P$ such that $\pi_*(P)=a$?

We will see that the answer to this question is negative in general. Moreover, for a given $a\in H^0(G,Q)$ we construct an element $o(a)\in H^2(G,A)$ such that $o(a)=0$ iff there exists a $(M\xto{\d} L)$-bitorsor $P$ such that $\pi_*(P)=a$.

\begin{Le} Let $a\in  H^0(G,Q)$. Then the set 
$$L_a=\{x\in L| \pi(x)=a\}$$
is a $G$-invariant subset of $L$.
\end{Le}

\begin{proof} Take $x\in L_a$. Then for any $g\in G$, we have 
$$\pi(x)=a=\, ^ga=\, ^g\pi(x)=\pi(^gx).$$
Thus $^gx\in L_a$.
\end{proof}
To construct $o(a)\in H^2(G,A)$ we consider the groupoid $\Gamma_a$. The set of objects of $\Gamma_a$ is $L_a$. If $x,y\in L_a$, then a morphism from $x$ to $y$ in  $\Gamma_a$ is a triple $(x,m,y)$, where $m\in M$ is such that $y=\d(m)x$. The composition law in $\Gamma_a$ is defined using the multiplication in $M$, that is 
$$(y,n,z)(x,m,y)=(x, nm,z).$$
Morever the action of $G$ on $L$ and $M$ induces a $G$-equivariant groupoid structure on $\Gamma_a$ simply by
$$^g(x,m,y):=(^gx,\,^gm,\, ^gy).$$
Thus $\Gamma_a$ is a $G$-equivariant groupoid. 

\begin{Pro} \label{31.20.08.2020}
\begin{itemize}

\item [i)] $\Gamma_a$ is a  $G$-bouquet.

\item [ii)] For any object $x\in L_a$, one has a canonical isomorphism
$$A\cong Aut_{\Gamma_a}(x),$$
which sends $\alpha\in A$ to the triple $(x,\alpha,x):x\to x$.

\item [iii)] Let us fix an object $x$. The corresponding group extension (\ref{canext})
 $$0\to A\xto{\kappa} B_{\Gamma_a}^x\xto{p} G\to 0$$
 has the following description. Elements of $B_{\Gamma_a}^x$ are pairs $(m,g)$, where
 $g\in G$, and $m\in M$ is an element for which
 $$\d(m) \, ^gx=x$$
 where the group structure is given by
 $$(m,g)(n, h)=(m ^gn, gh).$$
 Moreover, the maps $\kappa$ and $p$ are defined by
 $$p(m,g)=g, \ \kappa(\alpha)=(\alpha,1).$$
 \item [iv)] The set sections of $p$ in the extension iii) are in one-to-one correspondence with functions $\psi:G\to M$ such that 
$$\d(\psi(g)) \, ^gx=x.$$
Moreover, the set section of $p$ corresponding to a function $\psi$ is a group homomorphism iff
$$\psi(gh)=\psi(g) \, ^g\psi(h).$$

\item [v)] Assume $x,y\in L_a$. Then $y=\d (n) x$ for an element $n\in M$. One has the following commutative diagram of group extensions
$$\xymatrix{0\ar[r] & A \ar[r] \ar[d]_\id &B_{\Gamma_a}^x \ar[d]_\eta \ar[r]& G\ar[r] \ar[d]_\id &0\\
0\ar[r] & A \ar[r] &B_{\Gamma_a}^y \ar[r]& G\ar[r] &0\\
}$$
where $\eta(m,g)=(nm\,^g(n\1),g)$. 
\end{itemize}
\end{Pro}

\begin{proof}
\begin{itemize}
\item [i)] Since $\pi:L\to Q$ is surjective, $L_a\not = \emptyset$. Moreover, for any $x,y\in L_a$ we have $yx\1\in Ker(\pi)$. Thus, there is an element $m\in M$ such that $\d(m)=yx\1$. It follows that $(x,m,y)$ is a morphism in $\Gamma_a$. Thus $\Gamma_a$ is connected and therefore it is a $G$-bouquet.

\item [ii)] By definition, a morphism $x\to x$ is of the form $(x,\alpha,x)$ such that $\alpha\in M$ and $x=\d(m)x$, or equivalently $m\in A$ and the result follows.

\item [iii)] By definition, the elements of the group $B_{\Gamma_a}^x$ are pairs $(g,\chi)$, where $g\in G$ and $\chi:\,^gx\to x$ is a morphism of $\Gamma_a$. Hence, the result is a direct consequence of the description of morphisms in $\Gamma_a$.

\item [iv)] A section of $p$ looks like $G\ni g\mapsto (\psi(g),g)\in B_{\Gamma_a}^x$, where $\psi:G\to M$ is a map. The equlity in the question follows from iii). The last statement is a direct consequence of the description of the group law of $B_{\Gamma_a}^x$.

\item [v)] First off, we need to show that $\eta$ is a well-defined map. In fact, take $(m,g)\in B_{\Gamma_a}^x$. Thus $\d (m) ^gx=x$. We also have
$$\d (nm^g(n\1))^gy=\d(n)\d(m)\, ^g(\d(n\1)^g\d(n)^gx=\d(n) \d(m)^gx=\d(n) x=y$$
Thus $\eta$ is well-defined. The rest is obvious.
\end{itemize}
\end{proof}

\begin{De} For an element $a\in H^0(G,Q)$, we let $o(a)\in H^2(G,A)$ be the class of the extension in iii) of Proposition \ref{31.20.08.2020}.
\end{De}

By the part v) of Proposition \ref{31.20.08.2020}, this class is independent of the choice of the element $x\in L_a$. 

\begin{Le} Let $a\in H^0(G,Q)$. Then  $o(A)=0$ iff there is a  $(M\xto{\d} L)$-bitorsor $P$ such that $\pi_*(P)=a$.
\end{Le}
\begin{proof} By the part iv) of  Proposition \ref{31.20.08.2020}, $o(a)=0$ iff there exists an $x\in L_a$ and a function $\psi:G\to M$ such that both equalities in part iv) of Proposition \ref{31.20.08.2020} hold. Comparing these conditions with equalities (\ref{f=1coc}) and (\ref{detau}), we see that  $(\psi,x)\in Z^1(G,(M\xto{\d}L))$. By Proposition \ref{26.19.08.2020} we have a bitorsor $P$ with the expected properties. 

\end{proof} 
%For any $g\in G$, we have 
%$$\pi(x)=a=\, ^ga=\, ^g\pi(x)=\pi(^gx).$$
%Hence, there is a morphism  $\lambda_g=(^gx, \gamma(g), x)$ from $^gx$ to $x$ in the %groupoid $\Gamma_a$, for a map $\gamma:G\to M$. This implies that
%$$x=\d(\gamma(g))\,^gx.$$
%Consider the corresponding 
%$$\vartheta_{g,\lambda_g}(\alpha)=\lambda_g\circ\, ^g (x,\alpha,x) \circ \lambda\1,$$ %where $\alpha\in A.$
%Since the composition in $\Gamma_a$ is induced by  the multiplication in $M$ and $A$ %is the central subgroup of $M$, we obtain 
%\begin{equation}\label{aaa}
%\vartheta_{g,\lambda_g}(\alpha)=\,^g\alpha.
%\end{equation}
%We consider $\Gamma_a$ as a $G$-bouquet over the crossed module ${\bf %A}(A)=(A\xto{\iota} Aut(A))$. Since $A$ is commutative, the map $\phi:G\to Aut(A)$ %defined in Section \ref{23.16.08.20} is a homomorphism and the equality \ref{aaa} %shows that $\phi$ coincides with the given action of $G$ on $A$. In the virtue of %Corollary \ref{13.19.08.2020}, $\Gamma_a$ defines an element in $H^2(G,A)$, which we %denote by $o(a).$

%One easily sees that the action of $G$ on $L$ and $M$ induces an   which is a %$G$-bouquet. Clearly $Aut_{\Gamma_a}(x)\cong A$ for any $x\in L_a$. Hence we obtain a %canonical element $o(a)\in H^2(G,A)$.
%
{\bf Remark}. Starting with a crossed extension
$$0\to A\to M\xto{\d}\, L\xto{\pi}\, Q\to 0$$
Borovoi constructed an exact sequence 
\begin{align*}1\longrightarrow H^1(G,A)\longrightarrow H^1(G,M\to L)\longrightarrow H^0(G,Q)\longrightarrow H^2(G,A) \\
\longrightarrow H^2(G,M\to L)\longrightarrow H^1(G,Q)--->H^3(G, \psi_3(A)),\end{align*}
see \cite[ (2.20.2)]{borovoi}. So our consideration in this section is to understand the exactness of this exact sequence at $H^1(G, M\to L)$ in terms of bitorsors and bouquets.
%Question: what about Maclane's obstruction theory? Also, the last term of that exact sequence
\subsection{Weak bitorsors over crossed modules} 
In this section we introduce a weak version of a bitorsor over  a crossed module and show that for so called faithful crossed modules both notions coincide.

\begin{De}\label{wdefbitor}
Let $P$ be a left $G$-set. Assume also that $M$ acts on $P$ on the left and right such that
$$(mp)n=m(pn), \ ^g(pm)=\,^gp\, ^gm,  \quad ^g(mp)=\,^gm\,^gp, \quad m, n\in M, g\in G, p\in P.$$
We will say that $P$ is a weak $(M\xto{\d} L)$-bitorsor, if $M$ is a right $M$-torsor and  there is given a map $\alpha:P\to L$ such that
$$mp=p\,  ^{\xi(p)}m$$
holds for all $m\in M$, where $\xi(p)=\alpha(p)^{-1}.$
\end{De}
%The unique element $m$ for which  $q=pm$ is  Thus for $p,q\in P$ we have $[p\diagdown q]\in M$ and $p[p\diagdown q]=q$.

By Proposition \ref{bitisweak} any bitorsor over $(M\xto{\d} L)$ is also a weak bitorsor. We are interested to what extent the converse is also true. 

\begin{De} We will say a crossed module $(M\xto{\d} L)$ is faithful if the condition $^\tau m=\,^{\tau '} m$ for all $m\in M$ implies $\tau=\tau'$. 
\end{De}

In other words, $(M\xto{\d} L)$ is faithful iff the induced homomorphism
$$\rho:L\to Aut(M), \quad \tau\mapsto \alpha_\tau,$$
is injective, where $\alpha_\tau(m)=\,^\tau m$, $\tau\in L$ and $m\in M$.

\begin{Pro} If $(M\xto{\d} L)$ is a faithful crossed module, then any weak $(M\xto{\d} L)$-bitorsor is also a $(M\xto{\d} L)$-bitorsor.
\end{Pro}

\begin{proof} We have to show two identities

i) $\xi(p)=\d(n)\xi(pn)$ for all $p\in P$ and $n\in N$.

ii) $\xi(^xp)=\,^x\xi(p)$ for all $p\in P$  and $x\in G$.
\end{proof}

Take any $m\in M$. We have
\begin{align*}p \,^{\xi(p)}m&=mp \\ &= m(pn)n\1\\&=pn \, ^{\xi(pn)}mn\1\\&=p\, ^{\d(n)\xi(pn)}m.
\end{align*}
Since $P$ is a right torsor, we obtain $^{\xi(p)}m=\, ^{\d(n)\xi(pn)}m$. Since this holds for all $m\in M$ and $(M\xto{\d} L)$ is  faithful, we obtain i). 

For ii) we proceed as follows. Since $mp=p\,^{\xi(p)}m$, we apply the action by $x\in G$ to obtain
$$^x(mp)=\, ^xp\, ^x(^{\xi(p)}m).$$
Or, equivalently, $^xm^xp=\,^x p\, ^{^x{\xi(p)}}(^xm).$ 
On the other hand, we have $^xm^xp=\,^x p\, ^{\xi(^xp)}(^xm).$ Comparing these expressions and using the same reasoning as in the proof of i), we obtain ii).

\end{document}